\documentclass[10pt]{amsart}
\usepackage{amsmath}
\usepackage{amsthm}
\usepackage{amsfonts}
\usepackage{mathrsfs}
\usepackage{amssymb}
\usepackage{amscd}
\usepackage{pgf}
\usepackage{tikz}
\usetikzlibrary{cd}
\usepackage{amscd}

\usepackage{ytableau}

\expandafter\let\csname ver@amsthm.sty\endcsname\relax
\let\theoremstyle\relax

\usepackage[utf8]{inputenc}
\usepackage[
    breaklinks,
    colorlinks,
    citecolor=teal,
    linkcolor=teal,
    urlcolor=teal,
    pagebackref=true,
    hyperindex
    ]{hyperref}
\usepackage{fancyhdr}
\usepackage[
    hscale=0.7,
    vscale=0.75,
    headheight=13pt,
    centering,
    ]{geometry}
\usepackage{amsmath}
\usepackage{amsthm}
\usepackage{amssymb}
\usepackage{mathtools}
\usepackage{stmaryrd}
\usepackage{mathdots}
\usepackage[all,cmtip]{xy}
\usepackage{xcolor}
\usepackage{framed}
\usepackage{pgffor}
\usepackage[capitalize]{cleveref}

\theoremstyle{plain}

\newtheorem{theorem}{Theorem}[section]
\newtheorem{proposition}[theorem]{Proposition}
\newtheorem{lemma}[theorem]{Lemma}
\newtheorem{corollary}[theorem]{Corollary}

\newtheorem{conjecture}[theorem]{Conjecture}
\newtheorem{question}[theorem]{Question}

\theoremstyle{definition}

\newtheorem{definition}[theorem]{Definition}

\newtheorem{remark}[theorem]{Remark}
\newtheorem{example}[theorem]{Example}

\numberwithin{figure}{section}

\numberwithin{equation}{section}

\makeatletter

\def\@myMR[#1 #2]{\relax\ifhmode\unskip\spacefactor3000 \space\fi
  \MRhref{#1}{MR\,#1}}
\renewcommand\MR[1]{\@myMR[#1 ]}
\renewcommand{\MRhref}[2]{{\tiny%
  \href{http://www.ams.org/mathscinet-getitem?mr=#1}{#2}}%
}
\renewcommand*{\backref}[1]{}
\renewcommand*{\backrefalt}[4]{%
    \tiny%
    ({
    \ifcase #1 not cited%
          \or cit.\ on p.~#2%
          \else cit.\ on pp.~#2%
    \fi%
    })\\[-.6em]}

\def\maketitle{\par
  \@topnum\z@ 
  \@setcopyright
  \thispagestyle{empty}
  \ifx\@empty\shortauthors \let\shortauthors\shorttitle
  \else \andify\shortauthors
  \fi
  \@maketitle@hook
  \begingroup
  \@maketitle
  \toks@\@xp{\shortauthors}\@temptokena\@xp{\shorttitle}%
  \toks4{\def\\{ \ignorespaces}}
  \edef\@tempa{%
    \@nx\markboth{\the\toks4
      \@nx\MakeUppercase{\the\toks@}}{\the\@temptokena}}%
  \@tempa
  \endgroup
  \c@footnote\z@
    \renewcommand{\footnoterule}{%
      \kern -3pt
      \hrule width \textwidth height .5pt
      \kern 2pt
    }
  {
    \renewcommand\thefootnote{}
    \vspace{-2em}
    \footnote{
      \par\vspace{-1.2em}\noindent
      \def\@footnotetext##1{\noindent{\footnotesize##1}\par}%
      \let\@makefnmark\relax  \let\@thefnmark\relax
      \ifx\@empty\@date\else \@footnotetext{\@setdate}\fi
      \ifx\@empty\@subjclass\else \@footnotetext{\@setsubjclass}\fi
      \ifx\@empty\@keywords\else \@footnotetext{\@setkeywords}\fi
      \ifx\@empty\thankses\else \@footnotetext{%
        \def\par{\let\par\@par}\@setthanks}%
      \fi
    }
    \addtocounter{footnote}{-1}
  }
  \@cleartopmattertags
}

\def\@adminfootnotes{\@empty}

\def\@settitle{\begin{center}%
  \baselineskip14\p@\relax
    \bfseries
\Large
  \@title
  \end{center}%
}

\def\@setauthors{%
  \begingroup
  \def\thanks{\protect\thanks@warning}%
  \trivlist
  \centering\footnotesize \@topsep30\p@\relax
  \advance\@topsep by -\baselineskip
  \item\relax
  \author@andify\authors
  \def\\{\protect\linebreak}%
  \large{\authors}%
  \ifx\@empty\contribs
  \else
    ,\penalty-3 \space \@setcontribs
    \@closetoccontribs
  \fi
  \endtrivlist
  \endgroup
}

\def\@setaddresses{\par
  \nobreak \begingroup
\footnotesize
  \def\author##1{\end{minipage}\hskip 1sp \begin{minipage}{.5\textwidth}\raggedright%
    ~\\[2em]{\bf##1}\\[.5em]%
  }%
  \interlinepenalty\@M
  \def\address##1##2{\begingroup
    {\ignorespaces##2}\endgroup\\[.5em]}%
  \def\curraddr##1##2{\begingroup
    \@ifnotempty{##2}{\nobreak\indent\curraddrname
      \@ifnotempty{##1}{, \ignorespaces##1\unskip}\/:\space
      ##2\par}\endgroup}%
  \def\email##1##2{\begingroup
    \@ifnotempty{##2}{\nobreak\indent
      \@ifnotempty{##1}{, \ignorespaces##1\unskip}
      \ttfamily##2\par}\endgroup}%
  \def\urladdr##1##2{\begingroup
    \def~{\char`\~}%
    \@ifnotempty{##2}{\nobreak\indent\urladdrname
      \@ifnotempty{##1}{, \ignorespaces##1\unskip}\/:\space
      \ttfamily##2\par}\endgroup}%
  \setlength{\parindent}{0pt}%
  \vfill%
  {
  \begin{minipage}{0mm}
  \addresses
  \end{minipage}
  }
  \endgroup
}

\renewcommand{\author}[2][]{%
  \ifx\@empty\authors
    \gdef\authors{#2}%
    \g@addto@macro\addresses{\author{#2}}%
  \else
    \g@addto@macro\authors{\and#2}%
    \g@addto@macro\addresses{\author{#2}}%
  \fi
  \@ifnotempty{#1}{%
    \ifx\@empty\shortauthors
      \gdef\shortauthors{#1}%
    \else
      \g@addto@macro\shortauthors{\and#1}%
    \fi
  }%
}
\edef\author{\@nx\@dblarg
  \@xp\@nx\csname\string\author\endcsname}

\def\@secnumfont{\@empty}

\def\section{\@startsection{section}{1}%
  \z@{.7\linespacing\@plus\linespacing}{.5\linespacing}%
  {\large\bfseries\centering}}

\makeatother

\usepackage{enumerate}

\title{ Splitting the cohomology of Hessenberg varieties and $e$-positivity of chromatic symmetric functions}
\author{Alex Abreu and Antonio Nigro}

\address{
    Instituto de Matemática e Estatística\\
    Universidade Federal Fluminense\\
    Rua Prof. M. W. de Freitas, S/N\\
    24210-201 Niterói, Rio de Janeiro, Brasil
}

\email{alexbra1@gmail.com, antonio.nigro@gmail.com}

\usepackage{color}
\definecolor{forestgreen}{rgb}{0.13, 0.55, 0.13}

\newcommand{\col}{\colon}

\newcommand{\m}{\mathbf{m}}
\newcommand{\ind}{i}

\DeclareMathOperator{\Id}{Id}
\DeclareMathOperator{\Ima}{Im}
\DeclareMathOperator{\ch}{ch}

\DeclareMathOperator{\wt}{wt}

\DeclareMathOperator{\LLT}{LLT}

\DeclareMathOperator{\IncTree}{IncTree}

\DeclareMathOperator{\csf}{csf}
\DeclareMathOperator{\asc}{asc}
\DeclareMathOperator{\exc}{exc}

\DeclareMathOperator{\Poin}{Poin}
\DeclareMathOperator{\Gr}{Gr}

\newcommand{\C}{\mathbb{C}}  
\newcommand{\flag}{\mathcal{B}}

\newcommand{\h}{\mathcal{Y}}

\newcommand{\dw}{\dot{w}}

\newcommand{\bF}{\textbf{f}}
\newcommand{\bH}{\textbf{h}}

\begin{document}

\maketitle

\begin{abstract}
    For each indifference graph, there is an associated regular semisimple Hessenberg variety, whose cohomology recovers the chromatic symmetric function of the graph. The decomposition theorem applied to the forgetful map from the regular semisimple Hessenberg variety to the projective space describes the cohomology of the Hessenberg variety as a sum of smaller pieces. We give a combinatorial description of the Frobenius character of each piece. This provides a generalization of the symmetric functions attached to Stanley's local $\bH$-polynomials of the permutahedral variety to any Hessenberg variety. \par 
    
    As a consequence, we can prove that the coefficient of $e_{\lambda}$, where $\lambda$ is any partition of length $2$, in the $e$-expansion of the chromatic symmetric function of any indifference graph is non-negative. 
\end{abstract}

\tableofcontents

\section{Introduction}

The $\bH$-vector $(\bH_0,\ldots, \bH_n)$ of an abstract simplicial complex $\Delta$ of dimension $n-1$ is defined by the equation
\[
\sum_{i=0}^n \bF_{i-1}(q-1)^{n-i}=\sum_{i=0}^n \bH_iq^{n-i}
\]
where $\bF_i$ is the number of $i$-dimensional faces of $\Delta$ (with the usual convention  $\bF_{-1}=1$, unless $\Delta=\emptyset$). The $\bH$-polynomial of $\Delta$ is defined as
\[
\bH(\Delta;q)=\sum_{i=0}^n \bH_iq^i.
\]

If $\Sigma$ is a simplicial fan of $\mathbb{R}^n$ of dimension $n$, we associate to $\Sigma$ the simplicial complex of dimension $n-1$ where for each cone $F\in \Sigma$ we associate the set of its rays. We can define the $\bH$-vector of $\Sigma$ via the formula
\[
\sum_{i=0}^n \bF_{i}(q-1)^{n-i}=\sum_{i=0}^n \bH_iq^{n-i}
\]
where $\bF_i$ is the number of cones in $\Sigma$ of dimension $i$.

\begin{figure}[h]
\centering

\begin{minipage}{0.48\linewidth}
\centering
\begin{tikzpicture}[scale=2]
\draw[->] (0,0) -- (1,0);
\draw[->] (0,0) -- (0, 1);
\draw[->] (0,0) -- (-0.7, -0.7);
\draw[fill] (0,0) circle (1pt);
\end{tikzpicture}
        \caption{This fan has one $0$-dimensional cone, three $1$-dimensional cones and three $2$-dimensional cones. Its $\bF$-vector is $(1,3,3)$, while $(1,1,1)$ is its $\bH$-vector.}
\label{fig:P2}
\end{minipage}
\begin{minipage}{0.48\linewidth}
\centering
\begin{tikzpicture}[scale=2]
\draw[->] (0,0) -- (1,0);
\draw[->] (0,0) -- (0, 1);
\draw[->] (0,0) -- (0, -1);
\draw[->] (0,0) -- (-1, 0);
\draw[->] (0,0) -- (0.7, 0.7);
\draw[->] (0,0) -- (-0.7, -0.7);
\draw[fill] (0,0) circle (1pt);
\end{tikzpicture}
\caption{The $\bF$-vector is $(1,6,6)$, the $\bH$-vector is $(1,4,1)$.}
\label{fig:Bl_P2}
\end{minipage}
\end{figure}

The $\bH$-vector has a deep connection with geometry via toric varieties (see \cite{Stanley87}). The $\bH$-vector of a complete simplicial fan $\Sigma$ concides precisely with the (even) Betti numbers of the corresponding
toric variety $X(\Sigma)$. Indeed, the toric variety associated to the fan in Figure \ref{fig:P2} is $\mathbb{P}^2$, which has Betti numbers $(1,0,1,0,1)$ (the odd Betti numbers vanish). Likewise, the toric variety associated to the fan in Figure \ref{fig:Bl_P2} is the variety $\mathbb{L}_2$ obtained by blowing up $\mathbb{P}^2$  along the three coordinate points, which has Betti numbers $(1,0,4,0,1)$. We note that the fan in Figure \ref{fig:Bl_P2} is a refinement of the fan in Figure \ref{fig:P2}. This refinement induces a map between the corresponding toric varieties, which in this case is the natural map $\mathbb{L}_2\to \mathbb{P}^2$.\par 

Given a simplicial subdivision of the simplicial complex $\Delta$, Stanley (\cite{Stanley92}) introduced the notion of \emph{local $\bH$-vector}, which he used to strengthen some results around McMullen's $g$-conjecture. We will translate the definitions and results to fans, for the sake of simplicity. Let $\Sigma'$ be a simplicial refinement of a simplicial fan $\Sigma$. For each cone $F\in \Sigma$, let $\Sigma'_F$ denote the fan whose cones are the cones of $\Sigma'$ contained in $F$. Then, the local  $\bH$-polynomials $\ell_F(\Sigma'_F;q)$ are defined via the formula
\[
\bH(\Sigma'_F;q)=\sum_{W\preceq F}\ell_W(\Sigma'_W;q).
\]
Note that the local $\bH$-polynomial with respect to a cone $F$, only depends on $\Sigma'_F$ (explaining the adjective \emph{local}). Moreover, since the formula above is for any cone $F$, we see that it indeed determines $\ell_F(\Sigma'_F;q)$.

For example, in the refinement given by Figures \ref{fig:P2} and \ref{fig:Bl_P2}, we have that $\ell_F(\Sigma'_F;q)=1$ if $F=\{0\}$, $\ell_F(\Sigma'_F;q)=0$ if $\dim(F)=1$, and $\ell_F(\Sigma'_F;q)=q$ if $\dim(F)=2$.

These local $\bH$-vectors can be recovered by geometry (see \cite{ Stanley92}). Recall that for a finite dimensional graded vector space $W=\bigoplus_{i\in \mathbb{Z}} W_i$, its Hilbert-Poincaré polynomial is defined by $\Poin(W):=\sum_{i\in \mathbb{Z}}q^{\frac{i}{2}}\dim(W_i)$. Consider the map $f\col X(\Sigma')\to X(\Sigma)$. We have a stratification of this map given by the invariant subvarieties of $X(\Sigma)$, which are of the form $V(F)$ for a cone $F\in \Sigma$. For the case of simplicial toric varieties the decomposition theorem applied to $f$ has the form
\begin{equation}
\label{eq:decomp_local_h_vector}
H^*(X(\Sigma')) = \bigoplus_{F\in \Sigma}H^*(V(F))\otimes L_F,
\end{equation}
where $L_F$ are finite dimensional graded vector spaces. Then, the local $\bH$-polynomial $\ell_F(\Sigma'_F;q)$ is precisely the Hilbert-Poincaré polynomial of $L_F$.\par 

Indeed, the decomposition theorem applied to the map $f\col \mathbb{L}_2\to\mathbb{P}^2$ has the form
\begin{equation}
\label{eq:decomp_L2P2}
H^*(\mathbb{L}_2)=H^*(\mathbb{P}^2)\oplus (H^*(p_1)\otimes \mathbb{C}[-2])\oplus (H^*(p_2)\otimes \mathbb{C}[-2])\oplus (H^*(p_3)\otimes \mathbb{C}[-2]),
\end{equation}
where $\mathbb{C}[-2]$ is the graded vector space of dimension $1$  and degree $2$ (so its Hilbert-Poincaré polynomial is $q$) and  $p_1,p_2,p_3$ are the three coordinate points of $\mathbb{P}^2$, which are the invariant subvarieties corresponding to the three maximal cones of $\Sigma$. This recovers the local $\bH$-polynomials we already computed.

Stanley also considered generalizations of the $\bH$-polynomial as symmetric functions. Indeed, consider the action of $S_{n+1}$ on $\mathbb{R}^{n}$ permuting the vectors $\Vec{e}_1,\ldots, \Vec{e}_{n}$ and $\Vec{e}_{n+1}:=-(\Vec{e}_1+\ldots+\Vec{e}_{n})$. If $\Sigma$ is a fan invariant by the action of $S_{n+1}$, then we have an action of $S_{n+1}$ on $X(\Sigma)$, and can consider the graded Frobenius character of its cohomology. 

 Recall that if $W = \bigoplus_{\lambda\vdash n+1} L_{\lambda}\otimes W_\lambda$ is a graded $S_{n+1}$-module, where $L_{\lambda}$ is the simple $S_{n+1}$-module induced by $\lambda$ and each $W_\lambda$ is a finite dimensional graded vector space,  we define the graded Frobenius character of $W$ 
\[
\ch(W):=\sum_{\lambda}\Poin(W_{\lambda})s_{\lambda},
\]
where  $s_{\lambda}$ is the Schur symmetric function associated to $\lambda$. We also denote by $h_{j}$ the complete homogeneous symmetric function of degree $j$ and set $h_\lambda:=\prod h_{\lambda_i}$ for every partition $\lambda$.

For instance, the symmetric function $\bH$-polynomials of the fans in Figures \ref{fig:P2} and \ref{fig:Bl_P2} are $(q^2+q+1)h_3$ and $(q^2+q+1)h_3+qh_{2,1}$, respectively. Note that for an inclusion $S_d\to S_{n+1}$, we can compute the graded Frobenius characters with respect to the action of $S_d$ on $\mathbb{R}^{n}$.

The construction of the symmetric function generalization of the local $\bH$-polynomial is a little more involved. Let $F$ be a cone in $\mathbb{R}^{n}$ of dimension $d$, generated by $\Vec{e}_1,\ldots, \Vec{e}_d$. There is a natural action of $S_d$ on $\mathbb{R}^{n}$ permuting the rays $\Vec{e}_1,\ldots, \Vec{e}_d$ and fixing $\Vec{e}_{d+1},\ldots, \Vec{e}_n$, moreover this action fixes $F$. Consider $\Sigma'_F$ a subdivision of $F$ that is invariant by the action of $S_d$, in particular if $W,W'\preceq F$ are faces of $F$ of the same dimension, then $\Sigma'_W$ and $\Sigma'_{W'}$ are isomorphic. 

Moreover, for each $W\preceq F$ of dimension $k$ generated by  $\Vec{e}_{i_1},\ldots, \Vec{e}_{i_k}$, we have an action of $S_k$ on $W$. Since $\Sigma'_F$ is invariant by $S_d$, we have that $\Sigma'_{W}$ is invariant by $S_k$. In particular, we have that each $\Sigma'_W$ has a symmetric function generalization $\bH(\Sigma'_W;q,x)$ of the $\bH$-polynomial, which has degree in $x$ equal to the dimension $k$ of $W$. We define a symmetric function generalization of the local $\bH$-polynomials by the formula
\[
\bH(\Sigma'_F;q,x) = \sum_{k\leq d} \ell_{W_k}(\Sigma'_{W_k};q,x)h_{d-k},
\]
where $W_k$ is any face of dimension $k$. Note that $\bH(\Sigma'_{W_k};q,x)$ does not depend on the face of dimension $k$, hence $\ell_{W_k}(\Sigma'_{W_k};q,x)$ depends only on $k$.

In particular, we have an equivariant version of Equation \eqref{eq:decomp_local_h_vector}.  The summation is now over dimensions rather than  faces, $H^*(V(F))$ is considered as an $S_{\dim(F)}$-module, $L_F$ is considered as an $S_{n-\dim(F)}$-module and we have to make an induction from $S_{\dim(F)}\times S_{n-\dim(F)}$ to $S_n$. For example, Equation \ref{eq:decomp_L2P2} would become
\[
H^*(\mathbb{L}_2)=H^*(\mathbb{P}^2)\oplus \ind_{S_{2}\times S_1}^{S_3}(H^*(p_1)\otimes L),
\]
where $H^*(p_1)$ is a trivial $S_2$-module and $L$ is a trivial $S_1$-module in degree 2, so the Frobenius charater of $\ind_{S_{2}\times S_1}^{S_3}(H^*(p_1)\otimes L)$ is $qh_{2,1}$.

We will now focus on two fans. The fan associated to $\mathbb{P}^{n-1}$ and its first barycentric subdivision, which we will denote, respectively, by $\Sigma_n$ and $\Sigma'_n$ (Figures \ref{fig:P2} and \ref{fig:Bl_P2} show the case $n=3$). Note that $X(\Sigma'_n)$ is the so called \emph{permutahedral variety} (it is also the Losev-Manin space $\mathbb{L}_{n-1}$), also $\Sigma'_n$ is the fan of Coxeter chambers of $S_n$. The symmetric function generalization of the $\bH$-vector of $\Sigma'_n$ can be computed as in \cite[Equation (16)]{Stanley92}, and its generating function is
\[
F_1(x;q,z):=\frac{\sum_{n\geq 0}h_nz^n}{1-\sum_{n\geq 2} q[n-1]_qh_nz^n}. 
\]
The symmetric function generalization of the local $\bH$-vector of a cone $F\in \Sigma_n$ with respect to the refinement $\Sigma_n'$ only depends on the dimension of $F$ and not on $n$ (as the barycentric subdivision is recursive). 
Its generating function is  computed in \cite[Equation (17)]{Stanley92} as
\begin{equation}
    \label{eq:F2}
F_2(x;q,z):=\frac{1}{1-\sum_{n\geq 2}q[n-1]_qh_nz^n }.
\end{equation}

The generating functions $F_1, F_2$ arises in different settings. For example, they are the generating function of  Eulerian quasisymmetric functions  and multiset derangements (see \cite{ShareshianWachs_Eulerian}), respectively. More important to us, is the relation of $F_1$ with chromatic symmetric functions\footnote{See also Appendix \ref{app:Face_LLT} for a relation of the face vector with the LLT polynomial.}, which we are now going to explain.

Stanley's chromatic symmetric function of a graph $G$ is defined as
\[
\csf(G;x):=\sum_{\kappa\colon V(G)\to \mathbb{N}}\prod_{v\in V(\Gamma)} x_{\kappa(v)}
\]
where the sum runs through all proper colorings $\kappa\colon V(G)\to \mathbb{N}$, that is, $\kappa(v)\neq \kappa(w)$, whenever $v$ and $w$ are neighbors in $G$. This symmetric function is a generalization the chromatic polynomial of $G$.\par

In \cite{ShareshianWachs}, Shareshian--Wachs introduced a $q$-deformation of $\csf(G;x)$ defined for graphs with vertex set $[n]$:
\[
\csf_q(G;x,q) := \sum_{\kappa\colon [n]\to \mathbb{N}}q^{\asc_G(\kappa)}\prod_{j\in [n]}x_{\kappa(j)}
\]
where $\asc_G(\kappa) := |\{(i,j); i<j\leq \m(i), \kappa(i)<\kappa(j)\}$ is the number of $G$-ascents of the coloring $\kappa$. The function $\csf_q(G;x;q)$ is usually just a quasisymmetric function, but it is symmetric for certain classes of graphs.

An indifference graph $G$ is a graph that has vertex set $[n]:=\{1,\ldots, n\}$ for some $n\in \mathbb{N}$ and there exists a non-decreasing function $\m\colon[n]\to[n]$ satisfying $\m(i)\geq i$ for every $i\in [n]$ (these functions are called Hessenberg functions) such that $E(G)=\{\{i,j\};i<j\leq \m(i)\} $. We write $G_\m$ for the indifference graph induced by $\m$. If $G$ is an indifference graph, then $\csf_q(G;x,q)$ is a symmetric function. The path graph $P_n$ on $n$ vertices is the indifference graph associated to the Hessenberg function $\m(i)=\min(i+1,n)$ for $i\in [n]$.

Recall that the ring of symmetric functions has a natural involution $\omega$ that takes the complete homogeneous symmetric functions $h_j$ to the elementary symmetric functions $e_j$. In their work, Shareshian--Wachs noticed that the omega dual of $F_1$ is also the generating function of the chromatic quasisymmetric function of the path graph on $n$-vertices (see \cite{ShareshianWachs}), that is
\[
\sum \csf_q(P_n;x,q)z^n = \frac{\sum_{n\geq 0}e_nz^n}{1-\sum_{n\geq 2} q[n-1]_qe_nz^n}. 
\]

In particular the chromatic quasisymmetric function of the path graph on $n$ vertices is the omega dual of the Frobenius character of the cohomology of the permutahedral variety associated to $S_n$. 
The permutahedral variety is a special case of a Hessenberg variety, which is defined as follows. Given a Hessenberg function $\m$  and a $n\times n$ matrix $X$, the \emph{Hessenberg Variety} associated to $\m$ and $X$ is defined as
\[
\h_\m(X) := \{V_\bullet; XV_i\subset V_{\m(i)}\text{ for }i\in [n]\}\subset \flag,
\]
where $\flag$ is the flag variety of $\mathbb{C}^n$. When $X$ is regular semisimple, we say that $\h_\m(X)$ is a \emph{regular semisimple} Hessenberg variety. \par  

The cohomology of a regular semisimple Hessenberg variety has a natural structure of $S_n$-module given by Tymozcko \emph{dot action} (\cite{Tym08}). Based on the result for the path graph, Shareshian--Wachs conjectured that the chromatic quasisymetric function of $G_\m$ is equal to the omega dual of the graded Frobenius character of the cohomology of $\h_\m(X)$ for $X$ regular semisimple.  This conjecture has since been proven by Brosnan--Chow (\cite{BrosnanChow}) and Guay-Pacquet (\cite{GP}).\par 

Hessenberg varieties are usually not toric varieties, so they do not have an associated fan, but still, in analogy with the case of the permutahedral variety, one could think of the Frobenius character of its cohomology as the analogue of the symmetric function generalization of its $\bH$-polynomial. We could ask:
\begin{question}
\label{ques:local_h}
    What is the analogue for Hessenberg varieties of  the local $\bH$-polynomial and its symmetric function generalization? 
\end{question}

We recall that the local $\bH$-polynomial is defined for a subdivision, which in geometric terms means a map between toric varieties, and can be computed as the Poincaré polynomial of the local systems appearing in the decomposition theorem. We will construct an analogue of the local $\bH$-polynomial and its symmetric function generalization by applying the decomposition theorem to the forgetful map
\begin{align*}
    \h_\m(X)& \to \Gr(1,n)=\mathbb{P}^{n-1}\\
     V_\bullet & \mapsto V_1.
\end{align*}

We will just focus on the symmetric function version, as the polynomial can be easily recovered from the symmetric function. The first observation is that we need the decomposition to be $S_n$-equivariant. With that in mind, we take the time now to explain how to construct the action of $S_n$ on the cohomology of regular semisimple Hessenberg varieties. We follow Brosnan--Chow (\cite{BrosnanChow}) and \cite{AN_hecke} (see also \cite{ChaShvI}).

 We define the relative Hessenberg variety as
\[
\h_\m := \{(X,V_\bullet); XV_i\subset V_{\m(i)}\text{ for }i\in [n]\}\subset GL_n(\mathbb{C})\times \flag,
\]
we set $G:=GL_n(\mathbb{C})$ and we denote by $G^{rs}$ the open subset of $G$ consisting of regular semisimple invertible matrices, that is, invertible matrices with $n$ distinct eigenvalues. We have a natural map $\pi_1(G^{rs},X)\to S_n$ from the fundamental group of $G^{rs}$ to the symmetric group $S_n$ induced by the action of $\pi_1(G^{rs},X)$ on the set of eigenvalues of $X$. In particular, for every partition $\lambda$ of $n$ the associated simple $S_n$-module $L_{\lambda}$ induces a simple local system $\mathcal{L}_{\lambda}$ on $G^{rs}$. \par

Define $\h_\m^{rs}:=\h_\m\cap (G^{rs}\times \flag)$ and let $g\colon \h_{\m}^{rs} \to G^{rs}$ be the projection onto the first factor. Let $\mathbb{C}_{\h_{\m}^{rs}}$ be the trivial sheaf on $\h_{\m}^{rs}$ with fiber $\mathbb{C}$. Since the map $g$ is smooth (see \cite[Section 8.2]{BrosnanChow}) then the complex $Rg_*(\C_{\h_{\m}^{rs}})$ can be written as a sum of shifted local systems on $G^{rs}$ by Deligne's decomposition theorem (\cite{DLef69,Dhodge2}). By \cite[Corollary 119]{BrosnanChow}, we have that this decomposition is given by $Rg_*(\C_{\h_{\m}^{rs}}) = \bigoplus_{\lambda\vdash n} \mathcal{L}_{\lambda}\otimes W_\lambda(\m)$, where each $W_\lambda(\m)$ is a graded vector space. Alternatively, we say that $Rg_*(\C_{\h_\m^{rs}})$ is induced by a graded $S_n$-module (the $S_n$-module $\bigotimes_{\lambda\vdash n}L_\lambda\otimes W_{\lambda}(\m)$). This also means that the cohomology of $\h_{\m}(X):=g^{-1}(X)$ carries a natural action of $S_n$ (which is precisely Tymozcko's dot action, see \cite[Theorem 124]{BrosnanChow}).

 The Shareshian-Wachs conjecture can be reformulated as  
\[
\omega(\csf_q(G_\m;x,q))=\ch(Rg_*(\mathbb{C}_{\h_\m^{rs}})).
\]

The main proposal of this paper is to consider the forgetful map $f\colon \h_{\m}^{rs} \to G^{rs}\times \Gr(1,n)$ and to compute $Rf_*(\mathbb{C}_{\h_m^{rs}})$. Unlike $g$, the map $f$ is not smooth, so we will need the decomposition theorem of Beilinson, Bernstein and Deligne \cite{BBD}. First, we define the subvarieties of $G^{rs}\times \Gr(1,n)$ appearing in the decomposition theorem. Set
\begin{equation}
\label{eq:Hi}
\mathcal{H}_k := \{(X,V_1);\text{ there exists } V_{n-k}\text{ such that } V_1\subset V_{n-i}, XV_{n-k}=V_{n-k}\}\subset G^{rs}\times \Gr(1,n)    
\end{equation}
and $\mathcal{H}_k^\circ:=\mathcal{H}_k\setminus \bigcup_{j=k+1}^{n-1}\mathcal{H}_j$. Then, we have natural maps $\pi_1(\mathcal{H}_k^\circ, (X,V_1))\to S_k$ (see \cite[Proposition 7.4]{AN_parabolic}), in particular an $S_k$-module induces a local system on $\mathcal{H}_k^\circ$. By the classification of parabolic character sheaves on $G^{rs}\times \Gr(1,n)$ in \cite[Corollary 1.9]{AN_parabolic} (parabolic character sheaves were introduced by Lusztig in \cite{LusztigParabolicI}), we have that the decomposition theorem for $f$ can be written as
\begin{equation}
\label{eq:RF_decomp}
Rf_*(\C_{\h_\m^{rs}})=\bigoplus_{k=0}^{n-1}IC_{\mathcal{H}_k}(\mathcal{L}_{k})
\end{equation}
where each local system $\mathcal{L}_{k}$ is induced by a graded $S_k$-module $L_k$ (by restricting to a point $X\in G^{rs}$, we have that the decomposition theorem applied to the map  $\h_m(X)\to \mathbb{P}^{n-1}$ will give an $S_n$-equivariant decomposition). Our main result is a combinatorial description of the Frobenius character $\ch(L_k)$ of $L_k$. The symmetric function $\ch(L_k)$ is our answer to Question \ref{ques:local_h}.

To better state the combinatorial part, we recall some results of Stanley \cite{Stan95} and \cite{ANtree}. When $G$ is an indifference graph, we have the following result of Stanley, which express $\csf(G;x)$ in terms of the power-sum symmetric functions:

\begin{equation}
\label{eq:Stanley_p}
\csf(\m;x):=\csf(G_\m;x) = \sum_{\sigma\in S_{n,\m}}\omega(p_{\lambda(\sigma)})
\end{equation}
where $S_{n,\m}$ is the subset of the symmetric group $S_n$ consisting of permutations $\sigma$ such that $\sigma(i)\leq \m(i)$, and $\lambda(\sigma)$ is the cycle type of $\sigma$.

Equation \eqref{eq:Stanley_p} has also a $q$-analogue proved in \cite{ANtree}. We define $\rho_n$ as the symmetric function defined recursively by $[n]_qh_n=\sum_{i=0}^{n-1} h_i\rho_{n-i}$. The symmetric function $\rho_n$ is a $q$-deformation of $p_n$ (see \cite[Section 2]{ANtree} for other properties of $\rho_n$). 
For a permutation $\sigma\in S_{n}$, we write $\sigma = \tau_1\cdots \tau_k$ for its cycle decomposition, written in such a way that each cycle begins with its smallest element, and the cycles are ordered in an increasing way according to their smallest elements. We also write $\sigma^c$ to denote the permutation obtained from $\sigma$ by removing the parenthesis in the cycle decomposition\footnote{For example, if $\sigma=(154)(2639)(78)$ then  $\sigma^c=154263978$.}.  Finally, we define $\wt_\m(\sigma)$ as the number of $\m$-inversions of $\sigma^c$, that is, the number of pairs $(i,j)$ such that $i<j\leq \m(i)$ and $(\sigma^c)^{-1}(i)>(\sigma^c)^{-1}(j)$. The main result of \cite{ANtree} states that
\[
\csf_q(G_\m;x,q) = \sum_{\sigma=\tau_1\cdots\tau_k\in S_{n,\m}}q^{\wt_\m(\sigma)}\omega(\prod \rho_{|\tau_i|}).
\]

We are ready to enunciate the main theorem of this paper, that is proved in Section \ref{sec:geom}.

\begin{theorem}
\label{thm:main}
   Write 
   \[
   f_*(\mathbb{C}_{\h_\m^{rs}})=\bigoplus_{k=0}^{n-1}IC_{\mathcal{H}_k}(\mathcal{L}_k)
   \]
   where each $\mathcal{L}_k$ is induced by a graded $S_k$-module $L_k$. Then
   \[
   \omega(\ch(L_k)) = \sum_{\substack{\sigma=\tau_1\cdots \tau_{j}\in S_{n,\m}\\ |\tau_1|\geq n-k}} (-1)^{|\tau_1|-n
+k}q^{\wt_\m(\sigma)}h_{|\tau_1|-n+k}\omega(\rho_{|\tau_2|}\cdots \rho_{|\tau_j|}).
   \]
\end{theorem}

Based on the main theorem, we define symmetric functions $g_k(\m;x,q)$ which will be the main combinatorial object of this article.
  
  \begin{definition}
  \label{def:gk}
Let $\m\col [n]\to[n]$ be a Hessenberg function, for each integer $k<n$ we define
\[
g_k(\m;x,q):=\sum_{\substack{\sigma=\tau_1\cdots \tau_{j}\in S_{n,\m}\\ |\tau_1|\geq n-k}} (-1)^{|\tau_1|-n
+k}q^{\wt_\m(\sigma)}h_{|\tau_1|-n+k}\omega(\rho_{|\tau_2|}\cdots \rho_{|\tau_j|}).
\]
We also write $g_k(G_\m;x,q):=g_k(\m;x,q)$.
\end{definition}

\begin{example}
    When $\m\colon[n]\to[n]$ is given by $\m(i)=n$, that is, $G_\m$ is the complete graph with $n$ vertices, we have that $g_k(\m;x,q)=0$ for every $0<k<n$ and $g_0(\m;x,q)=(n-1)!_qe_0$. Indeed, the Hessenberg variety $\h_\m^{rs}$ is equal to $G^{rs}\times \flag$, which means that the forgetful map $\h_\m^{rs}\to G^{rs}\times \Gr(1,n)$ is a bundle whose fibers are isomorphic to the flag variety of $\mathbb{C}^{n-1}$.
\end{example}

\begin{example}
    When $\m = (2,4,4,5,6,6)$, we have the following expressions for $g_k(\m;x,q)$:
    \begin{align*}
        g_0(\m;x,q)& = (q+1)e_0, &  g_1(\m;x,q)& = qe_1, &     g_2(\m;x,q)& = (q^2+q)e_2, \\ 
        g_3(\m;x,q)& = q^2e_3, &     g_4(\m;x,q)& = (q^3+q^2)e_4,\\ 
        g_5(\m;x,q)& = \mathrlap{ (q^4+3q^3+q^2)e_{3,2} + (q^4+q^3+q^2)e_{4, 1} + (q^5+2q^4+2q^3+2q^2+q)e_5, }
    \end{align*}
    while the chromatic quasisymmetric function is given by
    \begin{align*}
    \csf_q(\m;x,q)=& (q^4+3q^3+q^2)e_{3, 2, 1} + (q^4+q^3+q^2)e_{3, 3} + (q^4+q^3+q^2)e_{4, 1, 1} + \\ 
    & (q^5+3q^4+4q^3+3q^2+q)e_{4, 2} + (2q^5+3q^4+3q^3+3q^2+2q)e_{5, 1} +\\ 
    & (q^6+2q^5+2q^4+2q^3+2q^2+2q+1)e_{6}.
    \end{align*}
\end{example}

Theorem \ref{thm:main} has the following direct corollary.
\begin{corollary}
 The symmetric functions $g_k(\m;x,q)$ are Schur-positive for every Hessenberg function $\m$.
\end{corollary}

Moreover, the chromatic quasisymmetric function can be recovered from the symmetric functions $g_k$.

\begin{theorem}
\label{thm:csf_from_g}
For every Hessenberg function $\m\col[n]\to[n]$ we have that
\[
\csf_q(G_;x,q)=\sum_{k=1}^n[n-k]_qe_{n-k}g_k(\m;x,q).
\]
\end{theorem}

When we consider the path graph $P_n$ with $n$-vertices, we can give more explicit characterizations of $g_k(P_n;x,q)$. First, we have that $g_k(P_n;x,q)=g_k(P_{k+1};x,q)$ for every $k<n$ (this is a consequence of Proposition \ref{prop:gk_invariant}). Second, the generating function for $g_k(P_{k+1};x,q)$ is precisely $\omega(F_2)$. In particular, this generating function gives a combinatorial interpretation of the coefficients of $g_k(P_n;x,q)$ in the monomial bases in terms of derangements (see Proposition \ref{prop:path_derangements}). The generating function above also proves that  $g_k(P_n;x,q)$ is $e$-positive.

The Stanley-Stembridge conjecture (\cite{StanStem}) states that $\csf(G_\m;x)$ is $e$-positive for every indifference graph $G_\m$.  Shareshian-Wachs (\cite{ShareshianWachs})  extended this conjecture to $\csf_q(G_\m;x,q)$), which is also equivalent to a previous conjecture of Haiman (\cite{Haiman}, see also \cite{AN_haiman}). Via Theorem \ref{thm:csf_from_g}, the $e$-positivity of $g_k(G_\m;x,q)$ for every $k=0,\ldots, n-1$ would imply the $e$-positivity of $\csf_q(G_\m;x,q)$ and, hence, the Stanley-Stembridge conjecture.

\begin{conjecture}
The symmetric functions $g_k(G_\m;x,q)$ are $e$-positive for every Hessenberg function $\m\col [n]\to[n]$ and every $k=0,\ldots, n-1$.
\end{conjecture}

We have the following result, proved in Section \ref{sec:comb}, about the $e$-positivity of $g_k(G_\m;x,q)$ (unfortunately, only for $q=1$).

\begin{theorem}
\label{thm:g_ek_positive}
The coefficient of $e_k$ in the $e$-expansion of $g_k(G_\m;x,1)$ is non-negative for every $k$ and every Hessenberg function $\m$.
\end{theorem}
As a direct corollary, we get the following $e$-positivity result for the chromatic symmetric function.
\begin{corollary}
The coefficient of $e_{a,b}$ in the $e$-expansion of $\csf(\m;x)$ is non-negative  for every Hessenberg function $\m$.
\end{corollary}

Finally, in the appendix, we will explain how to define functions $g_k(G;x)$ for every graph $G$ (together with a choice of a vertex $v$).

\section{Computing the Frobenius character $\ch(L_k)$.}
\label{sec:geom}

Let $G=GL_n(\mathbb{C})$ and identify $\mathbb{P}^{n-1}$ with $\Gr(1,\mathbb{C}^n)$. Consider the subvarieties $\mathcal{H}_k\subset G^{rs}\times \mathbb{P}^{n-1}$ and $\mathcal{H}_k^\circ$ defined in Equation \eqref{eq:Hi}, for $k=0,\ldots, n-1$.  For each pair $(X,V_1)\in \mathcal{H}_k^\circ$ there exists exactly one subspace $V'_{n-k}$ such that $V_1\subset V'_{n-k}$ and $XV'_{n-k}=V'_{n-k}$, namely $V'_{n-k}=\oplus_{j=0}^{n-k-1}X^jV_1$.

The varieties $\mathcal{H}_k$ are the parabolic Lusztig varieties in $G^{rs}\times \mathbb{P}^{n-1}$, see \cite[Example 3.17]{AN_parabolic}. Also, by  \cite[Sections 4 and 5]{AN_parabolic}  these parabolic Lusztig varieties are the ones that appear when we apply the decompostion theorem to the map $\h_\m^{rs}\to G^{rs}\times\mathbb{P}^{n-1}$. Moreover, \cite[Sections 4 and 5]{AN_parabolic} also give a characterization of the local systems that can appear.\par 

 We recall the characterization of the local systems. First, we recall how to construct the  natural map $\pi_1(\mathcal{H}_k^\circ, (X,V_1))\to S_k$ (\cite[Proposition 7.4]{AN_parabolic}). For a vector space $V$ we denote by $\flag(V)$ the flag variety of $V$. If $X$ is an invertible matrix and $V\subset \mathbb{C}^n$ is an invariant subspace, then $X$ induces a map $X\colon \mathbb{C}^n/V\to \mathbb{C}^n/V$, which we will abuse notation and call it $X$ as well.
By \cite[Proposition 7.4]{AN_parabolic}, the following map 
\[
\{(X,V_1,F_\bullet); (X,V_1)\in \mathcal{H}_k^\circ, F_\bullet \in \flag(\mathbb{C}^n/V'_{n-k}), XF_\bullet=F_\bullet\}\to \mathcal{H}_k^\circ
\]
is a $S_k$-Galois cover, which induces the map $\pi_1(\mathcal{H}_k^\circ, (X,V_1))\to S_k$. Each local systems $\mathcal{L}_k$ appearing in Equation \eqref{eq:RF_decomp} is a local system on $\mathcal{H}_k^\circ$ induced by a representation of $S_k$.

One of the main objects of this section will be the Hessenberg varieties over $\mathcal{H}_k^\circ$, that we now define. For each Hessenberg function $\m_k\col[k]\to[k]$ we define
\[
\h_{\m_k, \mathcal{H}_k^{\circ}}^{rs} = \{(X, V_1, F_{\bullet}); (X,V_1)\in \mathcal{H}_k^\circ, F_\bullet \in \flag(\mathbb{C}^n/V'_{n-k}), XF_j\subset F_{\m_k(j)}\},
\]
where $V'_{n-k}=\bigoplus_{j=0}^{n-k-1}X^jV_1$. Also note that since $XV'_{n-k}=V'_{n-k}$ we have that $X$ naturally induces a map $X\colon \mathbb{C}^n/V'_{n-k}\to \mathbb{C}^n/V'_{n-k}$.
\begin{proposition}
\label{prop:csf_H}
   Let $f_{\m_k}\col \h_{\m_k,\mathcal{H}_k^\circ}\to\mathcal{H}_k^\circ$ be the forgetful map, then $R(f_{\m_k})_*(\mathbb{C}_{\h_{\m_k,\mathcal{H}_k^\circ}})$ is a graded local system on $\mathcal{H}_k^\circ$ induced by a graded $S_k$-module whose character is equal to $\omega(\csf_q(G_{\m_k}))$.
\end{proposition}
\begin{proof}
  It is clear that the fiber $f_{\m_k}^{-1}(X,V_1)$ is isomorphic to the Hessenberg variety $\h_{\m_k}(X)$.   The result follows from Brosnan--Chow proof of the Shareshian--Wachs conjecture (\cite{BrosnanChow})
\end{proof}

Let $\m \col [n]\to[n]$ be a Hessenberg function. Consider the map
\begin{align*}
f\col \h_{\m}^{rs}&\to G^{rs}\times \mathbb{P}^{n-1}\\
 (X,V_\bullet)&\mapsto (X,V_1).
\end{align*}
By Equation \eqref{eq:RF_decomp}, we have that $Rf_*(\C_{\h_\m^{rs}})=\bigoplus_{k=0}^{n-1}IC_{\mathcal{H}_k}(\mathcal{L}_{k})$ where each $\mathcal{L}_k$ is a graded local system on $\mathcal{H}_k^\circ$ induced by a graded $S_k$-module $L_k$. The main goal of this section is to give a combinatorial description of $\ch(L_k)$.

We begin by stratifying the map $f$. The idea is to intersect $\h_{\m}^{rs}$ with (union of) Schubert cells as in \cite{Tym07}. Each cell can be described as follows. For each $k=0,1,\ldots, n-1$, set $J_k:=\{n-k+1,\ldots, n-1\}$ and define the set of minimum right coset representatives 
\[
(S_n)^{J_k}:=\{\sigma\in S_n; \sigma(i)<\sigma(i+1)\text{ for every $i\in J_k$}\}.
\]
For a permutation $\sigma\in S_n$, we set
\[
r_{ij}(\sigma) :=|\{k;k\leq i,\sigma(k)\leq j\}|.
\]
 
 Let $(X,V_1)$ be a pair in $\mathcal{H}_k^\circ$, we set $V_j':=\bigoplus_{i=0}^{j-1}X^iV_1$ for $j=1,\ldots, n-k$ (note that $XV_{n-k}'=V_{n-k}'$). For each $k=0,\ldots, n-1$ and each permutation $\sigma \in (S_n)^{J_k}$, we define the cell 
\[
(\h_{\m,k,\sigma}^{\circ})^{rs} := \left\{(X, V_\bullet)\in (\h_{\m})^{rs};\begin{array}{l} (X,V_1)\in \mathcal{H}_k^\circ, \dim V_i'\cap V_j = r_{ij}(\sigma)\\\text{ for }i=1,\ldots, n-k\text{ and }j=1,\ldots, n\end{array}\right\}.
\]

\begin{remark}
\label{rem:union_schubert}
Recall the definition of the Schubert cell. For a permutation $\sigma\in S_n$ and a full flag $F_\bullet$ the Schubert cell is defined as 
\[
\Omega_{\sigma}^\circ(F_\bullet) := \{V_\bullet; \dim F_i\cap V_j = r_{ij}(\sigma)\text{ for }i=1,\ldots, n\text{ and }j=1,\ldots, n\}.
\]
 For a partial flag variety $V'_\bullet = V_0'\subset V_1'\subset \ldots V'_{n-k}$ and a permutation $\sigma\in (S_n)^{J_k}$, we define the cell
 \[
 \Omega_{\sigma,n-k}^\circ(V'_\bullet)=\{V_\bullet; \dim V_i'\cap V_j := r_{ij}(\sigma)\text{ for }i=1,\ldots, n-k\text{ and }j=1,\ldots, n\},
 \]
 which is the union of Schubert cells
 \[
 \Omega_{\sigma,n-k}^\circ(V'_\bullet) = \bigsqcup_{\sigma'\in (S_1^{n-k}\times S_k)\sigma }\Omega_{\sigma'}^\circ(\overline{V}'_\bullet)
 \]
 where $\overline{V}'_\bullet$ is any completion of $V'_\bullet$ to a full flag. So the cell $(\h_{\m,k,\sigma}^\circ)^{rs}$ is the intersection of $(\h_\m^\circ)^{rs}$ with the cell $\Omega_\sigma (V_\bullet')$.\par 
 
 Moreover, for a full flag $F_\bullet=(F_0\subset F_1\subset \ldots \subset F_{n})$, we define
 \[
 \Omega_{\sigma, n-k}^\circ(F_\bullet):=\Omega_{\sigma, n-k}^\circ(F_0\subset F_1\subset \ldots \subset F_{n-k}).
 \]
\end{remark}

Each permutation $\sigma \in (S_n)^{J_k}$ is determined by its restriction $\sigma|_{[n-k]}\col [n-k]\to [n]$. Indeed, we have that $\sigma(i+1)>\sigma(i)$ for every $i=n-k+1,\ldots, n-1$. It is sometimes more intuitive to consider only the function $\tau\col [n-k]\to [n]$ associated to $\sigma$, in this case we write $\h_{\m,k,\tau}$ instead of $\h_{\m,k,\sigma}$. Also note that
\[
r_{ij}(\sigma)=|\{k;k\leq i, \tau(k)\leq j\}|=:r_{ij}(\tau)
\]
whenever $i\leq n-k$.

We have a characterization of the injective functions $\tau\col [n-k]\to [n]$ satisfying $\tau(1)=1$ and $\tau(i+1)<\m(\tau(i))$ in terms of increasing subtrees of the graph $G_\m$. We denote by $\IncTree_{n-k}(G_\m)$ the set of increasing trees with $n-k$ vertices contained in $G_\m$ and with root at vertex $1$. 

\begin{lemma}
\label{lembijection_sigma_T}
Let $k$ be a nonnegative integer, then there is a bijection 
\[
f\col \left\{
\sigma\in (S_n)^{J_k};\begin{array}{c}
     \sigma(i+1)\leq \m(\sigma(i))\text{ for every } i=1,\ldots, n-k-1\\
     \sigma(1)=1
\end{array}
\right\}
\leftrightarrow
\IncTree_{n-k}(G_\m).
\]
\end{lemma}
\begin{proof}
We use Stanley's bijection. We let $T$ be the tree whose vertices are $\sigma(1), \sigma(2),\ldots, \sigma(n-k)$ and there is an edge connecting vertices $\sigma(j)$ and $\sigma(i)$, with $\sigma(j)>\sigma(i)$ if $i$ is the largest index smaller than $j$ such that $\sigma(i)< \sigma(j)$ (such an $i$ always exists because $\sigma(1)=1<\sigma(j)$ for every $j>1$). In particular the tree $T$ is increasing as for every vertex $\sigma(j)$ there exists only one edge connecting $\sigma(j)$ with a vertex $\sigma(i)$ smaller than $\sigma(j)$.

To prove that $T$ is a subtree of $G_\m$, it is enough to check that every $i$ as above satisfies $\m(\sigma(i))\geq \sigma(j)$. However, we know that $\sigma(i+1)\leq \m(\sigma(i))$ and by construction either $\sigma(i+1)=\sigma(j)$ or $\sigma(i+1)>\sigma(j)$, in either case we have that $\sigma(j)<\m(\sigma(i))$.

Every increasing tree $T$ on $n-k$ vertices comes from a permutation $\sigma\in (S_n)^{J_k}$, we now prove that if $T$ is a subtree of $G_\m$, then $\sigma(i+1)<\m(\sigma(i))$ for every $i=1,\ldots, n-k-1$. If $\sigma(i)>\sigma(i+1)$ there is nothing to do. Otherwise, if $\sigma(i)<\sigma(i+1)$ then there is an edge connecting $\sigma(i)$ and $\sigma(i+1)$ which must be an edge of $G_\m$, which implies that $\sigma(i+1)<\m(\sigma(i))$.
\end{proof}

We will abuse notation and write $\IncTree_{n-k}(G_\m)$ to mean:
\begin{itemize}
\item the set of increasing subtrees of $G_\m$ with $n-k$ vertices;
\item the set of injective functions $\tau\col[n-k]\to[n]$ such that $\tau(1)=1$ and $\tau(i+1)\leq \m(\tau(i))$ for every $i=1,\ldots, n-k-1$;
\item the set of permutations $\sigma\in (S_n)^{J_k}$ such that $\sigma(1)=1$ and $\sigma(i+1)\leq \m(\sigma(i))$ for every $i=1,\ldots, n-k-1$.
\end{itemize}
Moreover, we write $\wt_\m(\sigma)$ to denote the number of $\m$-inversions of $\sigma$, that is pairs $i<j$ such that $\sigma(j)< \sigma(i)\leq \m(\sigma(j))$. We will also write $\wt_\m(\tau)$ and $\wt_\m(T)$ to mean the number of $\m$-inversions of the permutation $\sigma$ associated to $\tau$ or $T$.

If $\tau(1)=1$ (which will always be the case going forward), we let $j_1<j_2<\ldots < j_k$ be integers such that 
\begin{equation}
\label{eq:j_image_tau}
    \{j_1,\ldots, j_k\}=[n]\setminus \Ima(\tau),
\end{equation} and define $\m\setminus \tau$ as the Hessenberg function $\m\setminus \tau\col [k]\to [k]$ defined as
\[
(\m\setminus \tau) (i)= \max\{i_0; j_{i_0}\leq \m(j_i)\}.
\]
In particular the graph $G_{\m\setminus \tau}$ is obtained from the graph $G_{\m}$ by removing the vertices in $\Ima(\tau)$. We have the following characterization of the cell $(\h_{\m,k,\tau}^\circ)^{rs}$.

\begin{proposition}
\label{prop:properties_hmktau}
   The following properties hold.
   \begin{enumerate}
       \item The cell $(\h_{\m,k,\tau}^{\circ})^{rs}$ is nonempty only if $\tau(i+1)\leq \m(\tau(i))$ for every $i=1,\ldots, n-k-1$ and $\tau(1)=1$.
       \item We have a stratification 
       \[
       \h_{\m}^{rs}=\bigsqcup_{\substack{k=0,\ldots, n-1\\ \tau\col[n-k]\hookrightarrow[n]}} (\h_{\m,k,\tau}^\circ)^{rs}.
       \]
       \item If $\tau$ satisfies that $\tau(1)=1$ and $\tau(i+1)<\m(\tau(i))$ for every $i=1,\ldots, n-k-1$, then the map $(\h_{\m,k,\tau}^\circ)^{rs}\to \mathcal{H}_k^\circ$ factors through a map
       \[
       (\h_{\m,k,\tau}^\circ)^{rs}\to \h_{\m\setminus\tau,\mathcal{H}_k^{\circ}}^{rs}
       \]
       whose fibers are isomorphic to the affine space $\mathbb{A}^{\wt_\m(\tau)}$.
   \end{enumerate}
\end{proposition}
\begin{proof}
  Let us prove item (1). First, note that since $V_1=V_1'$ we must have that $\dim V_1'\cap V_1=1$ and hence $(\h_{\m,k,\tau}^{\circ})^{rs}$ is empty if $\tau(1)\neq 1$. 
   Assume by contradiction that there exist an integer $i\in \{1,\ldots, n-k-1\}$ such that $\tau(i+1)>\m(\tau(i))$ (in particular $\tau(i+1)>\tau(i)$), and a pair $(X,V_\bullet) \in (Y_{\m,k,\tau}^{\circ})^{rs}$. Note that
   \[
   r_{i+1,j}(\tau)=\begin{cases}
   r_{ij}(\tau)& \text{ if }\tau(i+1)>j\\
   r_{ij}(\tau)+1& \text{ if }\tau(i+1)\leq j,
   \end{cases}
   \]
   so we have the following equalities
   \begin{align*}
       r_{i-1,\tau(i)}(\tau)+1&=r_{i,\tau(i)}(\tau)=r_{i+1,\tau(i)}(\tau),\\
       r_{i,\tau(i+1)-1}(\tau)&=r_{i+1,\tau(i+1)-1}(\tau),
   \end{align*}
   which implies
   \begin{align*}
   \dim V'_{i-1}\cap V_{\tau(i)}  +1 &= \dim  V'_i\cap V_{\tau(i)} = \dim  V'_{i+1}\cap V_{\tau(i)},\\
   \dim  V'_{i}\cap V_{\tau(i+1)-1} &=  \dim  V'_{i+1}\cap V_{\tau(i+1)-1} .
   \end{align*}
   However, we have that $XV_{\tau(i)}\subset V_{\m(\tau(i))}\subset V_{\tau(i+1)-1}$ and $XV'_i\subset V'_{i+1}$. Hence $X( V'_i\cap V_{\tau(i)})\subset   V'_{i+1}\cap V_{\tau(i+1)-1} $. We now claim that $X( V'_i \cap V_{\tau(i)}) \not\subset V'_i$. Indeed, if that were that case, we would have that $ V_i'\cap V_{\tau(i)}\subset  V'_i \cap X^{-1}V'_i = V'_{i-1}$ and hence $\dim V'_i\cap V_{\tau(i)}= \dim V'_{i-1}\cap V_{\tau(i)}$, a contradiction. This means that $X( V'_i\cap V_{\tau(i)})\not\subset V'_i$, and since  $X( V'_i\cap V_{\tau(i)})\subset   V'_{i+1}\cap V_{\tau(i+1)-1} $, we have that $V_{\tau(i+1)-1}\cap V'_{i+1} \neq V_{\tau(i+1)-1}\cap V'_i$, a contradiction. This finishes the proof of item (1).\par
   Item $(2)$ follows directly from the Schubert decomposition of the flag variety and Remark \ref{rem:union_schubert}. \par
   The proof of item (3) is a little more involved, so we will postpone it for now.
\end{proof}

Proposition \ref{prop:properties_hmktau} has the following corollary.
\begin{corollary}
\label{cor:ch_f*}
 We have that $Rf_*(\mathbb{C}_{\h_{\m}^{rs}})|_{\mathcal{H}_k^\circ}$ is a local system on $\mathcal{H}_k^\circ$ induced by a graded $S_i$-module whose character is 
 \[
 \sum_{\tau\in \IncTree_{n-k}(G_\m)}q^{\wt_\m(\tau)}\omega(\csf_q(\m\setminus\tau)).
 \]
\end{corollary}
\begin{proof}
 By base change we have that $Rf_*(\mathbb{C}_{\h_{\m}^{rs}})|_{\mathcal{H}_k^\circ} = Rf_*(\mathbb{C}_{\h_\m^{rs}\cap f^{-1}(\mathcal{H}_k^\circ)})$. By the stratification
 \[
 \h_\m^{rs}\cap f^{-1}(\mathcal{H}_k^\circ) = \bigsqcup_{\tau\in \IncTree_{n-k}(\m) }(\h_{\m,k,\tau}^{\circ})^{rs},
 \]
 by the fact that $\h_{\m\setminus \tau, \mathcal{H}_k^\circ}^{rs}\to \mathcal{H}_{k}^{\circ}$ is a proper smooth map whose fibers have no odd cohomology and by Proposition \ref{prop:properties_hmktau} item (3)  we have that 
 \[
 Rf_*(\mathbb{C}_{\h_\m^{rs}\cap f^{-1}(\mathcal{H}_k^\circ)})=\bigoplus_{\tau}R(f_{\m\setminus\tau})_*(\mathbb{C}_{\h_{\m\setminus \tau, \mathcal{H}_k^\circ}^{rs}})[-2\wt_\m(\tau)].
 \]
 The proof is complete after applying Proposition \ref{prop:csf_H} to the right-hand side.
 \end{proof}
We now state two lemmas that we will use in the proof of Proposition \ref{prop:properties_hmktau} item (3). 
\begin{lemma}
\label{lem:X_block}
 Let $X$ be a $n\times n$ matrix with block decomposition
 \[
 X= \left(\begin{array}{cc}
     X_1 & 0\\
     0 & X_2
 \end{array}
 \right), \text{ where } X_1 =  \left(\begin{array}{ccccc}
     0 & 0 & \cdots & 0 & a_1 \\
     1 & 0 & \cdots & 0 & a_2  \\
     0 & 1 & \cdots & 0 & a_3   \\
          \vdots \\
     0 & 0 & \cdots & 1 & a_{n-k}\\
 \end{array}\right)
 \]
 and let $u$ be an upper triangular matrix with diagonal entries $1$ and block decomposition
 \[
 u=(u_{i,j})_{i,j=1,\ldots, n}= \left(\begin{array}{cc}
     u_1 & u_2\\
     0 & \Id
 \end{array}
 \right).
 \]
  Then, for every $i=2,\ldots, n-k$ and $j=i+1,\ldots, n$ we have that
 \[
 (u^{-1}Xu)_{i,j} = u_{i-1,j}+P_{i,j}(u_{i',j'}; i'\geq i),
 \]
 where $P_{i,j}(u_{i',j'}; i'\geq i)$ is a polynomial in $u_{i',j'}$ for $i'=i, i+1,\ldots, n-k$ and $j'=i'+1,\ldots, n$. Moreover, 
 \[
 (u^{-1}Xu)_{i,j} = X_{i,j}
 \]
 for every $i=n-k+1,\ldots, n$ and $j=1,\ldots, n$.
\end{lemma}
\begin{proof}
  We have that 
  \[
  u^{-1}Xu = \left( \begin{array}{cc}
  u_1^{-1}X_1u_1 & u_1^{-1}(X_1u_2+u_2X_2)\\
   0 & X_2
  \end{array}\right),
  \]
  which is sufficient to prove that 
  \[
 (u^{-1}Xu)_{i,j} = X_{i,j}
 \]
 for every $i=n-k+1,\ldots, n$ and $j=1,\ldots, n$.  Moreover, the following properties hold.
  \begin{enumerate}
      \item The entry $(u_1^{-1})_{i,j}$, for $j>i$ is a polynomial in $u_{i',j'}$ for $i'\geq i$ and $(u_1^{-1})_{ii}=1$.
      \item The entry $(u_2X_2)_{i,j}$ is a linear combination of $u_{i',j}$ for $i'=i$.
      \item $(X_1u_1)_{i,j}$ and $(X_1u_2)_{i,j}$ are $u_{i-1,j}$ plus a linear combination of $u_{i',j'}$ for $i'\geq i$.
  \end{enumerate}
  The result follows.
\end{proof}

\begin{lemma}
\label{lemsigma_inversion}
Let $\m$ be a Hessenberg function, $k$ be a nonnegative integer, and $\sigma\in \IncTree(G_\m)$ be a permutation. If $i<j$ is an inversion of $\sigma$ such that $\sigma(i)>\m(\sigma(j))$ then $i-1<j$ is an inversion of $\sigma$.
\end{lemma}
\begin{proof}
  We begin by noticing that $i>1$, otherwise, $i<j$ would not be an inversion because $\sigma(1)=1$.
  We have that $\sigma(i)>\m(\sigma(j))$ and $\sigma(i)\leq \m(\sigma(i-1))$, then $\m(\sigma(i-1))>\m(\sigma(j))$ which implies that $\sigma(i-1)>\sigma(j)$ and hence $i-1<j$ is an inversion of $\sigma$.
\end{proof}


\begin{proof}[Proof of Proposition \ref{prop:properties_hmktau} item (3)]
Let $\sigma$ be the permutation in $(S_n)^{J_k}$ associated to $\tau$. For a pair $(X,V_1)\in \mathcal{H}_k^\circ$, recall the construction of the flag 
\begin{equation}
\label{eq:V'_bullet_from_V1}
V'_\bullet = (V'_0\subset V'_1\subset \ldots \subset V'_{n-k}),
\end{equation}
where $V'_j = \sum_{i=0}^{j-1}X^iV_1$. Moreover, for a triple  $(X,V_1, F_\bullet)\in \h_{\m\setminus \tau,\mathcal{H}_k^\circ}$ we construct the flag $F'_\bullet$ as
\begin{equation}
\label{eq:F'_bullet_from_F}
\{0\}= F'_0\subset  F'_1\subset F'_2\subset \cdots \subset F'_{n-k} \subset F'_{n-k+1} \subset \cdots\subset F'_n = \mathbb{C}^n,
\end{equation}
where $F'_{j}=V'_j$ for $j=1,\ldots, n-k$ and $F'_{n-k+i}/V'_{n-k} = F_i$. Note that  $V'_\bullet$ (respectively, $F'_\bullet$) varies depending on $(X,V_1)$ (respectively,  $(X,V_1, F_\bullet)$).

Let us prove that
$\h_\m(X)\cap \Omega_{\sigma}(F'_{\bullet})$ is isomorphic to $\mathbb{A}^{\wt_{\m}(\tau)}$. Let $F'_\bullet = g'B$ as an element of $GL_n/B$. We can choose $g'$ such that
\[
X' = g'^{-1}Xg' = \left(\begin{array}{cc}
     X'_1 & 0\\
     0 & X'_2
 \end{array}
 \right), \text{ where } X'_1 =  \left(\begin{array}{ccccc}
     0 & 0 & \cdots & 0 & a_1 \\
     1 & 0 & \cdots & 0 & a_2  \\
     0 & 1 & \cdots & 0 & a_3   \\
          \vdots \\
     0 & 0 & \cdots & 1 & a_{n-k}\\
 \end{array}
 \right)
\]
(the matrix in Lemma \ref{lem:X_block}). Indeed, we can choose $g'=(v_1,\ldots, v_{n-k}, v_{n-k+1},\ldots, v_{n})$ where $v_{i+1}=Xv_i$ for $i=1,\ldots, n-k-1$, $Xv_{n-k}\in \langle v_1,\ldots, v_{n-k-1}\rangle$ and $Xv_{n-k+j}\in \langle v_{n-k+1},\ldots, v_{n}\rangle$ for $j=1,\ldots, k$.

On the other hand, if $gB$ is the representative of $V_\bullet$, we have that $V_\bullet\in \Omega_{\sigma}(F'_\bullet)$ if and only if $g \in g' U^{\sigma}\dot{\sigma} B$, in particular, we can choose $g$ such that  $g=g'u\dot{\sigma}$ for some $u\in U^{\sigma}$. In this case, the condition $g^{-1}Xg\in \overline{B\dw_\m B }$ becomes
\begin{equation}
    \label{eq:Xbar_wm}
    \overline{X}:= \dot{\sigma}^{-1}u^{-1}X'u \dot{\sigma} \in \overline{B\dw_\m B}.
\end{equation}
The condition above is equivalent to $\overline{X}_{ij}=0$ for every $i,j=1,\ldots, n$ such that $i>\m(j)$. Since $\overline{X}_{ij} = (u^{-1}X'u)_{\sigma^{-1}(i),\sigma^{-1}(j)}$, we have that Equation \eqref{eq:Xbar_wm} is equivalent to $(u^{-1}X'u)_{i,j}=0$ for every $i,j=1,\ldots, n$ such that  $\sigma(i)>\m(\sigma(j))$. \par

Since $u\in U^{\sigma}\cong \mathbb{A}^{\ell(\sigma)}$ we can write $u = (u_{i,j})$ where $u_{i,i}=1$, $u_{i,j}=0$ if $i>j$ or $i<j$ is not an inversion of $\sigma$. That means, that the coordinates of  $\mathbb{A}^{\ell(\sigma)}$ are $u_{i,j}$ for $i<j$  an inversion of $\sigma$. Since $\sigma(\ell+1)>\sigma(\ell)$ for every $\ell=n-k+1,\ldots, n-1$, we have that $u\in U^{\sigma}$ has block decomposition as in Lemma \ref{lem:X_block}.

 First, we check that the condition $(u^{-1}X'u)_{i,j}=0$ for every $i,j=1,\ldots, n$ such that  $\sigma(i)>\m(\sigma(j))$ is trivially satisfied whenever $i>n-k$. If $i>n-k$ and $j\leq n-k$ then $(u^{-1}X'u)_{i,j}=0$ by Lemma \ref{lem:X_block}.\par 
 
 Now, we consider the case $i=n-k+i'_0$, $j=n-k+i'_1$ for positive integers $i'_0,i'_1$. Since $(u^{-1}X'u)_{n-k+i'_0,n-k+i'_1}=X'_{n-k+i'_0,n-k+i'_1}$ by Lemma \ref{lem:X_block}, we have that $(u^{-1}X'u)_{n-k+i'_0,n-k+i'_1}=0$ whenever $i'_0>(\m\setminus\tau)(i'_1)$, because we started with $(X,V_1,F_\bullet)$ in $\h_{\m\setminus\tau, \mathcal{H}_k^{\circ}}$. Let $j_1,\ldots, j_k$ be as in Equation \eqref{eq:j_image_tau}. Then, $\sigma(n-k+i')=j_{i'}$, moreover, we have that $j_{i'_0}>\m(j_{i'_1})$ if, and only if, $i'_0>(\m\setminus\tau)(i'_1)$. This means that $(u^{-1}X'u)_{n-k+i'_0,n-k+i'_1}=0$ whenever $j_{i'_0}>\m(j_{i'_1})$, which is to say, whenever $\sigma(n-k+i'_0)>\m(\sigma(n-k+i'_1))$. \par

 Let us assume that $i\leq n-k$. Moreover, since $\sigma \in \IncTree(G_\m)$ we have that $\sigma(i)>\m(\sigma(j))$ implies that $i\neq j+1$, moreover, if $i> j+1$ then $(u^{-1}X'u)_{i,j} = 0$ by Lemma \ref{lem:X_block}. So, we can further restrict ourselves to the case $i<j$, $2\leq i\leq n-k$ and $\sigma(i)>\m(\sigma(j))$ (recall that $\sigma(1)=1$).\par

By Lemma \ref{lem:X_block}, we have that $(u^{-1}Xu)_{i,j} = u_{i-1,j}+P(u_{i',j'}; i'\leq i)$ and by Lemma \ref{lemsigma_inversion} we have that $u_{i-1,j}$ is a coordinate of $\mathbb{A}^{\ell(\sigma)}$ for every $i,j$ satisfying  $i<j$ and $\sigma(i)>\m(\sigma(j))$. Hence, we have that $Y_\m(X)\cap \Omega_{\sigma}(V'_{\bullet})$ is an affine space of dimension $\ell(\sigma)-|\{(i,j);i<j, \sigma(i)>\m(\sigma(j)\}|=\wt_\m(\tau)$. This proves the claim that $\h_\m(X)\cap \Omega_\sigma(F'_\bullet)$ is isomorphic to $\mathbb{A}^{\wt_{\m}(\tau)}$.\par

We will now construct a map $(\h_{\m,k,\tau}^{\circ})^{rs}\to \h_{\m\setminus \tau,\mathcal{H}_k^{\circ}}$ such that the fiber over $(X,V_1,F_\bullet)$ is $Y_\m(X)\cap \Omega_{\sigma}(F'_{\bullet})$, where $F'_\bullet$ is constructed as in Equation \eqref{eq:F'_bullet_from_F}. \par 

For a point $(X,V_\bullet)\in (\h_{\m,k,\tau}^\circ)^{rs}$, we set  $V'_\bullet$ as in Equation \eqref{eq:V'_bullet_from_V1} for the pair $(X,V_1)$. Consider the map 
\begin{align*}
(\h_{\m,k,\tau}^{\circ})^{rs}&\to \h_{\m\setminus \tau,\mathcal{H}_k^{\circ}}\\
 (X,V_\bullet) &\to V'_{n-k}\subset V'_{n-k}+V_{j_1}\subset V'_{n-k}+V_{j_2}\subset \cdots\subset V'_{n-k}+V_{j_k},
\end{align*}
where $\{j_1,\ldots, j_k\}=[n]\setminus \Ima(\tau)$. The map is well defined because 
\[
\dim V_{j_i}\cap V'_{n-k} = r_{j_i,n-k}(\tau) = j_i - i,
\]
hence 
\[
\dim V'_{n-k}+V_{j_i} = n-k+j_i -(j_i-i) = n-k+i.
\]
Moreover, 
\[
X(V_{n-k}'+V_{j_i})\subset V_{n-k}'+V_{\m(j_i)} = V_{n-k}' + V_{j_{(\m\setminus \tau)(i)}},
\]
where the equality follows from the fact that $r_{\ell+1, n-k} =r_{\ell, n-k}+1$ whenever $\ell+1\notin\{j_1,\ldots, j_k\}$.

If $F_\bullet$ is the image of $(X,V_\bullet)$ and we set $F'_\bullet$ to be the flag as in Equation \eqref{eq:F'_bullet_from_F}, we would have that $F'_\bullet$ would be equal to
\[
V_1\subset V'_2\subset \cdots \subset V'_{n-k}\subset V'_{n-k}+V_{j_1}\subset V'_{n-k}+V_{j_2}\subset \cdots\subset V'_{n-k}+V_{j_k}.
\]
Then it is clear that $V_\bullet\in \Omega_\sigma(F'_\bullet)$. Hence, the preimage of $V'_\bullet$ is precisely $\h_{\m}(X)\cap \Omega_\sigma(V'_\bullet)$. This finishes the proof.
\end{proof}

\begin{lemma}
\label{lemLi_restricted_Lj}
Let $L_i$ be a graded $S_i$-module. Consider $\mathcal{L}_i$ be the induced local system in $\mathcal{H}_i^{\circ}$. If $j > i$, then $IC_{\mathcal{H}_i}(L_i)|_{\mathcal{H}_j^\circ}$ is a local system induced by a $S_j$-module $L_j$ and 
\[
\ch(L_j)=\ch(L_i)h_{j-i}.
\]
\end{lemma}
\begin{proof}
  It is enough to prove the Lemma when $L_i=\ind_{S_\lambda}^{S_i}$ where $\lambda=(\lambda_1,\ldots, \lambda_{\ell(\lambda)})$ is a composition of $i$ and $\ind_{S_\lambda}^{S_i}$ is the representation of $S_i$ induced by the trivial representation of $S_\lambda$, in particular $\ch(L_i)=h_{\lambda}$. Write $\lambda^{tot}_j=\lambda_1+\ldots +\lambda_j$ and consider the varieties
  \begin{align*}
      \widetilde{\mathcal{H}}_i&:=\{(X, V_1\subset V_{n-i}); XV_{n-i}=V_{n-i}, X\in GL_n^{rs}\},\\
      \mathcal{H}_{i,\lambda} &:= \left\{\begin{array}{l}(X, V_1\subset V_{n-i}\subset V_{n-i+\lambda_1^{tot}}\subset V_{n-i+\lambda^{tot}_2}\subset\ldots\subset V_{n-i+\lambda^{tot}_{\ell(\lambda)}}); \\XV_{n-i+\lambda^{tot}_k}=V_{n-i+\lambda^{tot}_k}, X\in G^{rs}\end{array}\right\}.
  \end{align*}
  We have that $\widetilde{\mathcal{H}}_i$ is smooth and the natural forgetful map $\widetilde{\mathcal{H}}_{i}\to\mathcal{H}_{i}$ is a normalization map, hence finite. This implies that the map $f_{i, \lambda}\col\mathcal{H}_{i,\lambda}\to \mathcal{H}_i$ is also finite and hence small. So we have that 
  \[
  IC_{\mathcal{H}_i}(L_i)=(f_{i,\lambda})_*(\mathbb{C}_{\mathcal{H}_{i,\lambda}}),
  \]
  indeed, for a point $(X,V_1)\in \mathcal{H}_i^\circ$, we have that each point of the fiber $f_{i,\lambda}^{-1}(X,V_1)$ can be identified with a partition $\{1,\ldots, i\}=A_1\sqcup\ldots \sqcup A_{\ell(\lambda)}$ with $|A_{k}|=\lambda_k$. This means that the action of $S_i$ on the cohomology of this fiber is isomorphic to $\ind_{S_{\lambda}}^{S_i}$.\par

  Therefore, we have that
  \[
  IC_{\mathcal{H}_i}(L_i)|_{\mathcal{H}_j^\circ}=(g_{i,\lambda}^j)_*(\mathbb{C}_{f_{i,\lambda}^{-1}({\mathcal{H}_j^\circ})}),
  \]
  where $g_{i,\lambda}^j\col f_{i,\lambda}^{-1}(\mathcal{H}_j^\circ)\to \mathcal{H}_j^\circ$ is the restriction of $f_{i,\lambda}$ to $f_{i,\lambda}^{-1}(\mathcal{H}_j^\circ)$.\par 
  
  We can give a explicit description of $ f_{i,\lambda}^{-1}({\mathcal{H}_j^\circ})$. Since $(X,V_1)\in \mathcal{H}_j^\circ$, we have that there exists exactly one space $V_{n-j}$ such that $XV_{n-j}=V_{n-j}$ and $V_1\subset V_{n-j}$, so 
  \[
  f_{i,\lambda}^{-1}({\mathcal{H}_j^\circ})= \left\{ \begin{array}{l}(X, V_1\subset V_{n-j}\subset V_{n-i}\subset V_{n-i+\lambda_1^{tot}}\subset\ldots\subset V_{n-i+\lambda^{tot}_{\ell(\lambda)}});\\XV_{n-j}=V_{n-j}, XV_{n-i+\lambda^{tot}_k}=V_{n-i+\lambda^{tot}_k}, X\in G^{rs}\end{array}\right\}.
  \]
  By the same observation as before, we have that there is a bijection between the  fiber $(g_{i,\lambda}^j)^{-1}(X,V_1)$ and the set of partitions $\{1,\ldots, j\}=A_0\sqcup A_1\sqcup\ldots \sqcup A_{\ell(\lambda)}$ with $|A_0|=j-i$ and $|A_k|=\lambda_k$. Hence $L_j$ is isomorphic to $\ind_{S_{\lambda_1}}^{S_j}$ where $\lambda_1=(j-i,\lambda_1,\ldots, \lambda_{\ell(\lambda)})$. So $\ch(L_j)=h_{\lambda_1}=h_{j-i}h_{\lambda}$.
\end{proof}

\begin{theorem}
\label{thm:main_csf_rec}
   Write 
   \[
   f_*(\mathbb{C}_{\h_\m^{rs}})=\bigoplus_{i=0}^{n-1}IC_{\mathcal{H}_i}(\mathcal{L}_i)
   \]
   where each $\mathcal{L}_i$ is induced by a graded $S_i$-module $L_i$. Then
   \[
   \omega(\ch(L_i)) = \sum_{\tau \in IncTree_{\geq n-k}(G_\m)}(-1)^{|\tau|-n+k}q^{\wt_\m(\tau)}h_{|\tau|-n+k}\csf_q(G_{\m}\setminus \tau).
   \]
\end{theorem}
\begin{proof}
Restricting both sides of 
\[
   f_*(\mathbb{C}_{\h_\m^{rs}})=\bigoplus_{i=0}^{n-1}IC_{\mathcal{H}_i}(\mathcal{L}_i)
\]
to $\mathcal{H}_j^{\circ}$ and using Corollary \ref{cor:ch_f*} and Lemma \ref{lemLi_restricted_Lj} we have that
\[
\sum_{\tau\in \IncTree_{n-j}(G_\m)}q^{\wt_\m(\tau)}\omega(\csf_q(\m\setminus\tau)) = \sum_{i\leq j}h_{j-i}\ch(L_i).
\]
A simple induction finishes the proof.
\end{proof}

\section{Combinatorics of $\ch(L_k)$}
\label{sec:comb}
In this section we prove some combinatorial results about $\ch(L_J)$. Recall that $\rho_n$ is defined by the recursion
\[
[n]_qh_n=\sum_{i=0}^{n-1}h_{i}\rho_{n-i},
\]
equivalently, we have that
\begin{equation}
    \label{eq:rho_eh}
    \omega(\rho_n)=\sum_{i=1
}^n (-1)^{n-i}[i]_qe_ih_{n-i}.
\end{equation}

Recall the definition of the function $g_k(\m;x,q)$ in Definition \ref{def:gk}:
\[
g_k(\m;x,q):=\sum_{\substack{\sigma=\tau_1\cdots \tau_{j}\in S_{n,\m}\\ |\tau_1|\geq n-k}} (-1)^{|\tau_1|-n
+k}q^{\wt_\m(\sigma)}h_{|\tau_1|-n+k}\omega(\rho_{|\tau_2|}\cdots \rho_{|\tau_j|}).
\]
We begin, with the following proposition that, together with Theorem \ref{thm:main_csf_rec}, proves Theorem \ref{thm:main}.

\begin{proposition}
   We have that
   \[
   g_k(\m;x,t)=\sum_{\tau \in IncTree_{\geq n-k}(G_\m)}(-1)^{|\tau|-n+k}q^{\wt_\m(\tau)}h_{|\tau|-n+k}\csf_q(G_{\m}\setminus \tau)
   \]
\end{proposition}
\begin{proof}
  We can write
  \begin{align*}
  g_k(\m;x,q)&=\sum_{\substack{\sigma=\tau_1\cdots \tau_{j}\in S_{n,\m}\\ |\tau_1|\geq n-k}} (-1)^{|\tau_1|-n +k}q^{\wt_\m(\sigma)}h_{|\tau_1|-n+k}\omega(\rho_{|\tau_2|}\cdots \rho_{|\tau_j|})\\
  &=\sum_{\tau\in \IncTree_{\geq n-k}(G_\m)}\sum_{\sigma=\tau_2\cdots \tau_{j}\in S_{n-|\tau|,\m\setminus \tau}} (-1)^{|\tau|-n +k}q^{\wt_\m(\tau)+\wt_{\m\setminus \tau}(\sigma)}h_{|\tau|-n+k}\omega(\rho_{|\tau_2|}\cdots \rho_{|\tau_j|})\\
  &=\sum_{\tau\in \IncTree_{\geq n-k}(G_\m)}(-1)^{|\tau|-n +k}q^{\wt_\m(\tau)}h_{|\tau|-n+k}\sum_{\sigma=\tau_2\cdots \tau_{j}\in S_{n-|\tau|,\m\setminus \tau}} q^{\wt_{\m\setminus \tau}(\sigma)}\omega(\rho_{|\tau_2|}\cdots \rho_{|\tau_j|})\\
  &=\sum_{\tau\in \IncTree_{\geq n-k}(G_\m)}(-1)^{|\tau|-n +k}q^{\wt_\m(\tau)}h_{|\tau|-n+k}\csf_q(\m\setminus \tau),\\
  \end{align*}
  where the last equality follows from \cite[Theorem 1.2]{ANtree}.
\end{proof}
The following proposition recovers the chromatic symmetric function from the functions $g_k$. 
\begin{proposition}
\label{prop:csf_from_g}
For every Hessenberg function $\m\col[n]\to[n]$ we have that
\[
\csf_q(\m;x,q)=\sum_{i=1}^n[n-i]_qe_{n-i}g_i(\m;x,q)
\]
\end{proposition}
\begin{proof}
This folows directly from Theorem \ref{thm:main}. Below, we give a pure combinatorial proof.\par 
By \cite[Theorem 1.2]{ANtree} we have
\begin{align*}
\csf_q(\m;x,q)=&\sum_{\sigma=\tau_1\cdots \tau_{j}\in S_{n,\m}}q^{\wt_\m{\sigma}}\omega(\rho_{|\tau_1|}\cdots \rho_{|\tau_j|})\\
    =&\sum_{\sigma=\tau_1\cdots \tau_{j}\in S_{n,\m}}q^{\wt_\m{\sigma}}\omega(\rho_{|\tau_1|})\omega(\rho_{|\tau_2|}\cdots \rho_{|\tau_j|}).
\end{align*}
Applying Equation \eqref{eq:rho_eh}, we have that
\begin{align*}
    \csf_q(\m;x,q)=&\sum_{\sigma=\tau_1\cdots \tau_{j}\in S_{n,\m}}q^{\wt_\m{\sigma}}\big(\sum_{i=1}^{|\tau_1|} (-1)^{|\tau_1|-i}[i]_qe_ih_{|\tau_1|-i}\big)\omega(\rho_{|\tau_2|}\cdots \rho_{|\tau_j|})\\
    =&\sum_{i=1}^n[i]_qe_i\sum_{\sigma=\tau_1\cdots \tau_{j}\in S_{n,\m}}(-1)^{|\tau_1|-i}q^{\wt_\m(\sigma)}h_{|\tau_1|-i} \omega(\rho_{|\tau_2|}\cdots \rho_{|\tau_j|})\\
    =& \sum_{i=1}^n[i]_qe_ig_{n-i}(\m;x,t).
\end{align*}
This finishes the proof.
\end{proof}

We now prove Theorem \ref{thm:g_ek_positive}.

\begin{proof}[Proof of Theorem \ref{thm:g_ek_positive}]
We begin by noting that the coefficient of $e_k$ in $h_{|\tau_1|-n+k}\omega(\rho_{|\tau_2|}\cdots \rho_{|\tau_j|})$ is
\[
[e_k](h_{|\tau_1|-n+k}\omega(\rho_{|\tau_2|}\cdots \rho_{|\tau_j|}))=\begin{cases}
  (-1)^{k-1}  & \text{ if }|\tau_1|=n,\\
    [k]_q& \text{ if } |\tau_1|=n-k\text{ and } |\tau_2|=k,\\
  0&\text{ otherwise}.
\end{cases}
\]
Hence the $e_k$-coefficient of $g_k(\m;x,q)$ is equal to 
\[
c_k(\m;q):=[k]_q\sum_{\substack{\sigma=\tau_1\tau_2\in S_{n,\m}\\ |\tau_1|=n-k,|\tau_2|=k}  }q^{\wt_{\m}(\sigma)} -\sum_{\sigma=\tau_1\in S_{n,\m}}q^{\wt_\m(\sigma)}
\]
We let $S_1$ be the set whose elements are permutations $(w_1,\ldots, w_n)$ of $(1,\ldots, n)$ such that $w_1=1$ and $w_{j+1}\leq \m(w_j)$ for every $j=1,\ldots, n-1$. Clearly, there is a bijection between the set $S_1$ and the set $\{\sigma \in S_{n,\m}; \sigma \text{ is a $n$-cycle} \}$. \par
We let $S_2$ to be the set of pairs of sequences $(w_1,\ldots, w_{n-k}),(z_1,\ldots, z_k)$ such that 
\[
w_1,\ldots, w_{n-k}, z_1,\ldots, z_k
\]
is a permutation of $(1,\ldots, n)$, $w_1=1$, $w_{j+1}\leq \m(w_j)$ for every $j=1,\ldots, n-k-1$, $z_{j+1}\leq \m(z_j)$ for every $j=1,\ldots, k-1$ and $z_1\leq \m(z_k)$.\par
We have a map $f\col S_2\to S_{n,\m}$ that takes a pair $\tau_1,\tau_2$ to the permutation $\sigma=\tau_1\tau_2$, moreover it is clear that $|f^{-1}(\sigma)|=k$ if $\sigma$ is in the image of $f$. Indeed,  $f^{-1}(f(\tau_1,\tau_2))$ is the set of all pairs $(\tau_1,\tau_2')$ where $\tau_2'$ is obtained from $\tau_2$ by rotating its entries.\par

 We have that $c_k(\m;1)=|S_2|-|S_1|$, therefore, to prove that $c_k(\m;1)$ is positive we construct a injective function $\Delta\col S_1\to S_2$. The function $\Delta$ is defined as follows.\par
 Let $w=(1,w_2,\ldots, w_n)$ be a sequence in $S_1$ and consider $j$ to be smallest non-negative integer satisfying  and $w_{n-j-k+1}\leq \m(w_{n-j})$. Note that $j=n-k$ satisfies the condition, so there exist at least one such $j$. We define
 \[
 \Delta(w)=(1,w_2,\ldots, w_{n-j-k},w_{n-j+1},\ldots, w_n),L^{j}(w_{n-j-k+1},\ldots, w_{n-j}),
 \]
 where $L$ is the left cyclic shift operator $L(a_1,\ldots, a_i):=(a_2,\ldots,a_i,a_1)$.\par
 \begin{example}
     Let $\m=(3,5,5,5,6,6)$, so $4\gg 1$, $6\gg 2,3,4$. Also, consider $k=3$. We have that 
     \[
     S_1=\left\{\begin{array}{l}(1, 2, 3, 4, 5, 6), (1, 2, 3, 5, 6, 4), (1, 2, 4, 5, 6, 3), (1, 2, 4, 3, 5, 6), (1, 2, 5, 6, 4, 3), \\ (1, 2, 5, 6, 3, 4),  (1, 3, 4, 5, 6, 2), (1, 3, 5, 6, 4, 2), (1, 3, 2, 4, 5, 6), (1, 3, 5, 6, 2, 4),\\ (1, 3, 2, 5, 6, 4), (1, 3, 4, 2, 5, 6)\end{array}\right\}
     \]
     and
     \[
     S_2 = \left \{ \begin{array}{l}((1, 2, 3), (4, 5, 6)), ((1, 2, 3), (5, 6, 4)), ((1, 2, 3), (6, 4, 5)), ((1, 2, 4), (3, 5, 6)),\\ ((1, 2, 4), (5, 6, 3)), ((1, 2, 4), (6, 3, 5)), ((1, 3, 2), (4, 5, 6)), ((1, 3, 2), (5, 6, 4)), \\ ((1, 3, 2), (6, 4, 5)), ((1, 3, 4), (2, 5, 6)), ((1, 3, 4), (5, 6, 2)), ((1, 3, 4), (6, 2, 5))\end{array}\right\}
     \]
     
     In the table below we have all elements in $S_1$ and their image through $\Delta$. In particular, $\Delta$ is a bijection in this case.
\begin{table}[h!]
\begin{tabular}{ccc}
           $w\in S_1$ &  $\Delta(w)$ &  \\
           \hline
$((1, 2, 3, \textcolor{red}{ 4, 5, 6})$& $((1, 2, 3), (4, 5, 6)))$ \\
$((1, 2, 3, \textcolor{red}{ 5, 6, 4})$& $((1, 2, 3), (5, 6, 4)))$\\
$((1, 2, 4, \textcolor{red}{ 5, 6, 3})$& $((1, 2, 4), (5, 6, 3)))$\\
$((1, 2, 4, \textcolor{red}{ 3, 5, 6)}$& $((1, 2, 4), (3, 5, 6)))$\\
$((1, 2, \textcolor{red}{ 5, 6, 4}, 3)$& $((1, 2, 3), (6, 4, 5)))$ & $6 > \m(3)$\\
$((1, 2, \textcolor{red}{ 5, 6, 3}, 4)$& $((1, 2, 4), (6, 3, 5)))$ & $6 > \m(4)$\\
$((1, 3, 4, \textcolor{red}{ 5, 6, 2})$& $((1, 3, 4), (5, 6, 2)))$\\
$((1, 3, \textcolor{red}{ 5, 6, 4}, 2)$& $((1, 3, 2), (6, 4, 5)))$ & $6 > \m(2)$\\
$((1, 3, 2, \textcolor{red}{ 4, 5, 6})$& $((1, 3, 2), (4, 5, 6)))$\\
$((1, 3, \textcolor{red}{ 5, 6, 2}, 4)$& $((1, 3, 4), (6, 2, 5)))$ & $6 > \m(4)$\\
$((1, 3, 2, \textcolor{red}{ 5, 6, 4})$& $((1, 3, 2), (5, 6, 4)))$\\
$((1, 3, 4, \textcolor{red}{ 2, 5, 6})$& $((1, 3, 4), (2, 5, 6)))$\\
     \end{tabular}
     \caption{The function $\Delta$ for $\m=(3,5,5,5,6,6)$ and $k=3$.\phantom{aaaaaaaa}}
     \label{tab:Delta}

\end{table}
 \end{example}
 Let us prove that $\Delta$ is injective. If $k=1$, then $\Delta(1,w_1,\ldots, w_n)=((1,w_2,\ldots, w_{n-1}),(w_n))$ so it is clear that $\Delta$ is injective. We will assume that $k\geq 2$. Let $(1,w_2,\ldots, w_{n-k}),(z_1,\ldots, z_k)$ be a pair in the image of $\Delta$ with at least two preimages $w$ and $w'$. That means, that there exists $j'$ and $j$ (we assume without loss that $j<j'$) such that
 \begin{align*}
 w:=&(1,\ldots, w_{n-k-j}, z_{k-j+1}, \ldots, z_{2k-j}, w_{n-k-j+1},\ldots, w_{n-k}),\\
 w':=&(1,\ldots, w_{n-k-j'}, z_{k-j'+1}, \ldots, z_{2k-j'}, w_{n-k-j'+1},\ldots, w_{n-k}).
 \end{align*}
 We write the indices of $z_j$ modulo $k$, that is $z_j:=z_{j\bmod k}$ for every $j\in \mathbb{Z}$.
 
 In what follows, for $a,b\in [n]$ we will write $a\gg b$ if $a>\m(b)$.
 We know that $z_{k-j+1}\leq \m(w_{n-k-j})$ and writing $j'-j=qk+r$ we have that (if $r>0$)
  \begin{multline*}
 z_{k-j+1}=z_{k-j'+1+r} \gg  w_{n-k-j'+r} \gg w_{n-k-j'+r+(k-1)} \gg \\ \gg w_{n-k-j'+r+2(k-1)} \gg \ldots \gg w_{n-k-j'+r+q(k-1)}=w_{n-k-j-q}.
 \end{multline*}
 Writing $\overline{w}_i:=w_{n-k-j'+r+i(k-1)}$, we have
 \[
 z_{k-j+1}\gg \overline{w}_0\gg \overline{w}_1\gg\ldots \gg \overline{w}_{q}=w_{n-k-j-q}
 \]
 Since $\m(w_{n-k-j-q})\geq w_{n-k-j-q+1}$ and $\overline{w}_{q-1}>\m(\overline{w}_{q})$ we have that $\overline{w}_{q-1}>w_{n-k-j-(q-1)}$, analogously, we have that $\overline{w}_{q-2}>w_{n-k-j-(q-2)}$ and so on until $\overline{w}_{0}>w_{n-k-j}$. This implies that $z_{k-j +1}>\m(\overline{w_0})\geq \m(w_{n-k-j})$ which is a contradiction with the fact that $w\in S_{1}$.\par
 If $r=0$, we have that
 \begin{multline*}
 z_{k-j+2}=z_{k-j'+2} \gg  w_{n-k-j'+1} \gg w_{n-k-j'+1+(k-1)} \gg \\ \gg w_{n-k-j'+r+2(k-1)} \gg \ldots \gg w_{n-k-j'+1+q(k-1)}=w_{n-k-j-q+1}. 
 \end{multline*}
By the same argument as above, we have that 
\[
z_{k-j+1}\gg w_{n-k-j'+1} \gg w_{n-k-j'+1+(k-1)}> w_{n-k-j}
\]
However, we also have that $z_{k-j+2}\leq \m(z_{k-j+1})$ which implies that
\[
z_{k-j+1} > w_{n-k-j'+1} \gg w_{n-k-j'+1+(k-1)} > w_{n-k-j},
\]
hence $z_{n-k-j}\gg w_{n-k-j}$, a contradiction.\par
\end{proof}

Unfortunately, the function $\Delta$ in the proof above is not weight preserving, so we can not prove that $c_k(\m;q)$ has non-negative coefficients.

\begin{corollary}
The coefficient of $e_{a,b}$ in the $e$-expansion of $\csf(\m;x)$ is non-negative  for every Hessenberg function $\m$.
\end{corollary}
\begin{proof}
  we have that the coefficient of $e_{a,b}$ in the $e$-expansion of $\csf_q(\m;x;q)$ is $[a]_qc_b(\m;q)+[b_q]c_a(\m;q)$ by Proposition \ref{prop:csf_from_g}. The result follows by Theorem \ref{thm:g_ek_positive}.
\end{proof}

We now prove some invariance results about the symmetric functions $g_k(\m;x,q)$ and extend the definition for the cases where $k\geq n$.

\begin{definition}
 Let $\m\col[n]\to[n]$ be a Hessenberg function and define $\m^{n'}\col[n'+n]\to[n'+n]$ by $\m^{n'}(i)=i+1$ for $i=1,\ldots, n'$ and $\m^{n'}(i)=\m(i)+n'$ for $i=n'+1,\ldots, n'+n$.  In particular the graph $G_{\m'}$ is obtained from $G_{\m}$ by appending a path of size $n'+1$ to the vertex $1$ of $G_{\m}$ (which will become the vertex $n'+1$ of $G_{\m'}$).
 \end{definition}

\begin{proposition}
\label{prop:gk_invariant}
 Let $n'$ be a positive integer. Then
 \[
 g_k(\m^{n'};x,t)=g_k(\m;x,t)
 \]
 for every $k<n$.
\end{proposition}
\begin{proof}
  Every increasing tree $\tau'$ of $G_{\m^{n'}}$ with at least $n'+1$ vertices consists of the path from $0$ to $n'+1$ joined with an increasing tree $\tau$ of $G_{\m}$. Furtermore it is clear that $\wt_{\m^{n'}}(\tau')=\wt_{\m}(\tau)$. The result follows from Definition \ref{def:gk}
\end{proof}

\begin{definition}
\label{def:gk_k>n}
Let $\m\col [n]\to[n]$ be a Hessenberg function. We extend the definition of $g_k(\m;x,t)$ for $k\geq n$ by
\[
g_k(\m;x,q)=g_k(\m^{n'};x,q)
\]
for some (equivalently, every) $n'$ such that $n+n'>k$.
\end{definition}
\begin{proposition}
\label{prop:gn-recursion}
   Let $\m\col[n]\to[n]$ be a Hessenberg function, then
   \[
   g_n(\m;x,q)=q\sum_{i=1}^n[i-1]_qe_ig_{n-i}(\m;x,t).
   \]
\end{proposition}
\begin{proof}
let $\m'=\m^{1}$. Then
\begin{align*}
    g_n(\m')=&\sum_{\substack{\sigma=\tau_1\cdots\tau_j\in S_{n+1,\m'}\\|\tau_1|=1}}q^{\wt_{\m'}(\sigma)}(-1)^0h_0\omega(\rho_{|\tau_2|}\cdots \rho_{|\tau_j|})+\\ &+\sum_{\substack{\sigma=\tau_1\cdots\tau_j\in S_{n+1,\m'} \\ |\tau_1|\geq 2}}q^{\wt_{\m'}(\sigma)}(-1)^{|\tau_1|-1}h_{|\tau_1|-1}\omega(\rho_{|\tau_2|}\cdots \rho_{|\tau_j|}) \\
    =&\csf_q(\m;x,t)+\sum_{\sigma=\tau_1\cdots\tau_j\in S_{n,\m'}}q^{\wt_{\m'}(\sigma)}(-1)^{|\tau_1|}h_{|\tau_1|}\omega(\rho_{|\tau_2|}\cdots \rho_{|\tau_j|})\\
    =&\sum_{i=1}^{n}[i]_qe_ig_{n-i}(\m;x,q)+\sum_{i=1}^ne_i\sum_{\sigma=\tau_1\cdots\tau_j\in S_{n,\m'}}q^{\wt_{\m'}(\sigma)}(-1)^{|\tau_1|-i-1}h_{|\tau_1|-i}\omega(\rho_{|\tau_2|}\cdots \rho_{|\tau_j|})\\
    =&\sum_{i=1}^{n}[i]_qe_ig_{n-i}(\m;x,q)- \sum_{i=1}^{n}e_ig_{n-i}(\m;x,q)\\
    =&q\sum_{i=1}^n[i-1]_qe_ig_{n-i}(\m;x,t).
\end{align*}

\end{proof}

\begin{theorem}
\label{thm:gk_generating_function}
   Let $\m\col[n]\to[n]$ be a Hessenberg function, then
   \[
   \sum_{k\geq 0}g_k(\m;x,t)z^k=\sum_{i=0}^{n-1}\frac{g_i(\m;x,t)z^i}{1-q\sum_{k\geq 2}[k-1]_qe_kz^k}(1-q\sum_{k= 2}^{n-i-1}[k-1]_qe_kz^k)
   \]
\end{theorem}
\begin{proof}
Denote by 
\[
G(z)=\sum_{k\geq 0}g_k(\m;x,t)z^k.
\]
By Proposition \ref{prop:gn-recursion} we have that
\begin{align*}
    G(z)(\sum_{j\geq1}q[j-1]e_jz^j)=&\sum_{k\geq 0}z^k(\sum_{j=1}^kq[j-1]_qe_jg_{k-j}(\m;x,t))\\
    =&\sum_{k=0}^{n-1} z^k(\sum_{j=1}^kq[j-1]_qe_jg_{k-j}(\m;x,t))+\sum_{k\geq n}z^k(\sum_{j=1}^kq[j-1]_qe_jg_{k-j}(\m;x,t))\\
    =&\sum_{k=0}^{n-1} z^k(\sum_{j=1}^kq[j-1]_qe_jg_{k-j}(\m;x,t))+\sum_{k\geq n}g_k(\m;x,t)z^k\\
    =&G(z)+\sum_{k=0}^{n-1} z^k(-g_k(\m;x,t)+\sum_{j=1}^kq[j-1]_qe_jg_{k-j}(\m;x,t))
\end{align*}
Then, we have that\footnote{checar indices dos somatorios com o enunciado}
\begin{align*}
G(z)(1-\sum_{j\geq1}q[j-1]e_jz^j)=&\sum_{k=0}^{n-1} z^k(g_k(\m;x,t)-\sum_{j=1}^kq[j-1]_qe_jg_{k-j}(\m;x,t))\\
                                 =&\sum_{k=0}^{n-1} z^kg_k(\m;x,t)(1-\sum_{j= 1}^{n-k-1}q[j-1]_qe_jz^j).
\end{align*}
This finishes the proof.
\end{proof}
Note that Theorem \ref{thm:gk_generating_function} says that $g_k(\m;x,t)$ for $k\geq n$  can be written in terms of 
\[
g_0(\m;x,t),g_1(\m;x,t),\ldots, g_{n-1}(\m;x,t).
\]
\begin{corollary}
 Let $\m\col[n]\to[n]$ be a Hessenberg function. If $g_1,\ldots, g_{n-1}$ are $e$-postive then $g_k(\m;x,t)$ is $e$-positive for every $k$.
\end{corollary}
\begin{proof}
This follows directly from Proposition \ref{prop:gn-recursion}.
\end{proof}

We end this section by computing the symmetric functions $g_k$ for the path graphs and relate them to derangements. We denote by $\bullet$ the graph with a single vertex. By Proposition \ref{prop:gk_invariant} and Definition \ref{def:gk_k>n}, we have that $g_k(P_n;x,q)=g_k(\bullet;x,q)$ for every $n$. The following direct corollary of Theorem \ref{thm:gk_generating_function} gives the generating function of $g_k(\bullet;x,q)$. Recall the generating function $F_2$ defined in Equation \eqref{eq:F2}

\begin{corollary}
We have the following equality:
\label{cor:gen_function_g_path}
   \[
   \sum_{k\geq 0} g_k(\bullet;x,t)z^k=\frac{1}{1-q\sum_{k\geq 2}[k-1]_qe_kz^k }=\omega(F_2(x;q,z)).
   \]
\end{corollary}
This is sufficient to prove the $e$-positivity of $g_k(\bullet;x,q)$.
\begin{corollary}
 For every $n$, we have that $g_k(P_n;x,q)$ is $e$-positive.
\end{corollary}

We now relate $g_k(\bullet;x,q)$ with derangements.
\begin{definition}
 Let $w=w_1w_2\ldots w_n$ be a word in the alphabet $\mathbb{P}$. A derangement of $w$ is a word $\overline{w}=\overline{w}_1\ldots \overline{w}_n$ such that every letter $i\in \mathbb{P}$ appears the same number of times in $w$ and in $\overline{w}$ and $w_j\neq \overline{w}_j $ for every $j=1,\ldots, n$. The number of excendances of $\overline{w}$ is the number of indices $j\in [n]$ such that $\overline{w}_j>w_j$ and is denoted as $\exc_w(\overline{w})$
\end{definition}
\begin{definition}
 For each partition $\lambda=(\lambda_1,\ldots, \lambda_{\ell(\lambda)})\vdash n$, we define $w_{\lambda}=1^{\lambda_1}2^{\lambda_2}\ldots \ell(\lambda)^{\lambda_{\ell(\lambda)}}$ and let
 \[
 c_{\lambda}(q):=\sum_{\overline{w}}q^{\exc_{w_\lambda}(\overline{w})}
 \]
 where the sum runs through all derangements of $w_\lambda$.
\end{definition}

\begin{proposition}
\label{prop:path_derangements}
   We have that
   \begin{equation}
   \label{eq:gk_clambda}
   g_k(\bullet;x,t)=\sum_{\lambda\vdash n}c_{\lambda}(q)m_{\lambda}.    
   \end{equation}
   \end{proposition}
\begin{proof}
 The right hand side of Equation \eqref{eq:gk_clambda} has the same generating function of $g_k(\bullet;x,t)$. See \cite{AskeyIsmail}, \cite{Stanley92} \cite{Kim_derangements} and \cite{ShareshianWachs_Eulerian}.
\end{proof}

\section{Further directions}
The chromatic quasisymmetric function of indifference graphs satisfy a linear relation called the \emph{modular law} that, together with the values on complete graphs, is enough to characterize the chromatic symmetric function (see \cite{AN}). The symmetric functions $g_k$ do not satisfy the modular law (one reason is that they do not have the property that $g_k(\m;x,q) = g_k(\m^t;x,q)$, where $\m^t$ is the transposed Hessenberg function of $\m$).
\begin{question}
Do the functions $g_i$ satisfy some modification/restriction of the modular law? 
\end{question}
\begin{question}
Is there an algorithm for computing the functions $g_k$ as the one in \cite{AN} to compute the chromatic quasisymmetric functions?
\end{question}

The chromatic quasisymmetric function is defined in the monomial basis of the symmetric functions. More than that, there are explicit combinatorial descriptions of the coefficients of $\csf_q(\m;x,q)$ in other basis (see \cite{ShareshianWachs}, \cite{Atha}, \cite{CHSS}). 
\begin{question}
Give a combinatorial description of the coeffients of the monomial-expasion of $g_k(\m;x,q)$.
\end{question}

\begin{question}
Give a combinatorial description of the coeffients of the Schur-expasion of $g_k(\m;x,q)$.
\end{question}

\appendix

\section{The $g$ symmetric functions for general graphs}

We can extend the definition of the symmetric functions $g$ to general graphs, although it still depends on the choice of a distinguished vertex. For a pair $(G,v_0)$ where $v_0$ is a vertex of $G$ and a set $S\subset E(G)$ we denote by $\lambda_0(S)=\lambda_{0}(G,v_0,S)$ the number of vertices in the component of $G\setminus S$ containg $v_0$ and by $\lambda(S)=\lambda(G,v_0,S)$ the partition induced by the number of vertices in the components of $G\setminus S$ that do not contain $v_0$.

\begin{definition}
\label{def:gk_general}
For a graph $G$ with $n$ vertices and  with a distinguished vertex $v_0$, we define
\[
g_{k}(G,v_0;x)=\sum_{S\subset E(G)}(-1)^{|S|-n+k+1}h_{\lambda_0(S)-n+k} p_{\lambda(S)}
\]
for every $k=0,\ldots, n-1$.
\end{definition}
\begin{remark}
\label{rem:multiple_edges}
We could make the definition above when $G$ has multiple edges as well. However, if $\underline{G}$ is the graph obtained by $G$ by removing the multiple edges, we have that $g_k(G,v_0;x)=g_k(\underline{G},v_0;x)$.
\end{remark}
We note that the sum in Definition \ref{def:gk_general} still makes sense when $k=n$, in that case we have the following result
\begin{proposition}
\label{prop:gn_pseudo}
If $n>0$, we have that
\[
\sum_{S\subset E(G)}(-1)^{|S|+1}h_{\lambda_0(S)} p_{\lambda(S)}+\sum_{k=0}^{n-1}e_{n-k}g_k(G,v_0;x) = 0.
\]
\end{proposition}
\begin{proof}
  The result follows from the equality
  \[
  \sum (-1)^ke_{n-k}h_{\lambda_0(S)-n+k} = 0.
  \]
\end{proof}

\begin{proposition}
   The following equality holds
   \[
   \csf(G;x) = \sum_{k=0}^{n-1} (n-k)e_{n-k}g_k(G,v_0;x).
   \]
\end{proposition}
\begin{proof}
  By \cite{Stan95} we have that
  \[
  \csf(G;x)= \sum_{S\subset E(G)}(-1)^{|S|}p_{\lambda_0(S)}p_{\lambda(S)}.
  \]
  On the other hand
  \begin{align*}
      \sum_{k=0}^{n-1} (n-k)e_{n-k}g_k(G,v_0;x) & = \sum_{k=0}^{n-1} (n-k)e_{n-k} \sum_{S\subset E(G)}(-1)^{|S|-n+k+1}h_{\lambda_0(S)-n+k} p_{\lambda(S)}\\
       &= \sum_{S\subset E(G)}(-1)^{|S|}p_{\lambda(S)}\sum_{k=0}^{n-1}(-1)^{-n+k+1}(n-k)e_{n-k}h_{\lambda_0(S)-n+k}\\
       &= \sum_{S\subset E(G)}(-1)^{|S|}p_{\lambda(S)}\sum_{k=0}^{\lambda_0(S)-1}(-1)^{1-\lambda_0(S)+k}(\lambda_0(S)-k)e_{\lambda_0(S)-k}h_{k}\\
       &=  \sum_{S\subset E(G)}(-1)^{|S|}p_{\lambda_0(S)}p_{\lambda(S)}.
  \end{align*}
  The result follows.
\end{proof}

\begin{proposition}[Deletion-Contraction]
If $e\in E(G)$ is an edge incident to $v_0$, then 
\[
g_{k}(G,v_0;x) =  g_{k}(G\setminus e,v_0;x)+g_{k}(G/e,v_0;x) 
\]
for every $k<n-1$ and
\[
g_{n-1}(G,v_0;x)=g_{n-1}(G\setminus e,v_0;x) - \sum_{k=0}^{n-2}e_{n-1-k}g_k(G/e,v_0;x).
\]
\end{proposition}
\begin{proof}
  First assume that $k<n-1$. Then, we can write
  \begin{multline*}
  g_{k}(G,v_0;x)=\sum_{\substack{S\subset E(G)\\ e \in S}}(-1)^{|S|-n+k+1}h_{\lambda_{0}(G,v_0,S)-n+k} p_{\lambda(G,v_0,S)}+\\ +\sum_{\substack{S\subset E(G)\\ e \notin S}}(-1)^{|S|-n+k+1}h_{\lambda_{0}(G,v_0,S)-n+k} p_{\lambda(G,v_0,S)}.
  \end{multline*}
  However, if $e\in S$, then $\lambda_{0}(G,v_0,S) = \lambda_{0}(G/e,v_0, S\setminus e)+1$ which implies $\lambda_{0}(G,v_0,S)-n+k = \lambda_{0}(G/e,v_0, S\setminus e) - (n-1)+k$. More so, if $e\notin S$, then $\lambda_{0}(G,v_0,S) = \lambda_{0}(G\setminus e,v_0, S)$. This means that
  \[
g_{k}(G,v_0;x) =  g_{k}(G\setminus e,v_0;x)+g_{k}(G/e,v_0;x).
\]
  
  If $k=n-1$. We have that 
  \[
  g_{n-1}(G,v_0;x)=\sum_{\substack{S\subset E(G)\\ e \in S}}(-1)^{|S|-n+k+1}h_{\lambda_{0}(G,v_0,S)-n+k} p_{\lambda(G,v_0,S)}+g_{n-1}(G\setminus e, v_0;x),
  \]
  by Proposition \ref{prop:gn_pseudo}, we get
  \[
  g_{n-1}(G,v_0;x)=-\sum_{k=0}^{n-2}e_{n-1-k}g_k(G/e,v_0;x)+g_{n-1}(G\setminus e, v_0;x),
  \]
  and the result follows.
\end{proof}

\section{LLT polynomial for the Path graph and face vector}
\label{app:Face_LLT}
 Consider the $S_n$-action on $\mathbb{R}^{n-1}$ permuting the vectors $e_1,\ldots, e_{n-1}$ and $e_n=-(e_1+\ldots e_{n-1}$). For a fan $\Sigma\subset \mathbb{R}^{n-1}$ that is invariant under this action, that is, if $\delta\in \Sigma$, then $\sigma(\delta)\in \Sigma$ for every $\sigma\in S_n$,  we can consider the graded $S_n$-module $C_\Sigma$ whose basis are the cones is $\Sigma$ with degree equal to codimension.

Let $\Sigma_0$ be the simplicial fan whose cones are $\langle e_i \rangle_{i\in S}$ for every $S\subsetneqq [n]$. Let $\Sigma_1$ be the first barycentric subdivison of $\Sigma_0$. Also let $P_n$ be the path graph on $n$-vertices.

\begin{proposition}
    We have that
    \[
    \ch(C_{\Sigma_1})=\omega(\LLT(P_n;x, q+1)).
    \]
\end{proposition}
\begin{proof}
    We have the following characterization of 
    \[
    \LLT(P_n;x, q+1)=\sum_{\mu\vDash n} q^{n- \ell(\mu)}e[\mu].
    \]
    
    Analogously, for each sequence $\underline{S}:= (S_1\subset S_2\subset \ldots S_d\subsetneqq [n])$ we have an associated cone $\delta_{\underline{S}}$ of dimension $d$. The composition of $\underline{S}$ is simply $\mu(\underline{S})=|S_1|+|S_2\setminus S_1|+\ldots |[n]\setminus S_d|$ which has length $d+1$, in particular, the codimension of $\delta_{\underline{S}}$ is equal to $n-1-d=n-\ell(\mu(\underline{S}))$. Clearly, the $S_n$-module generated by all the cones with fixed composition $\mu$ is isomorphic to the induced representation $\ind_{S_{\mu}}^{S_n}$, hence we have that
    \[
    \ch(C_{\Sigma_1})=\sum_{\mu\vDash n} q^{n- \ell(\mu)}h[\mu].
    \]
    The result follows.
    \end{proof}

\bibliographystyle{amsalpha}
\bibliography{bibli}

\end{document}